\documentclass[11pt]{amsart}
\usepackage{amsmath,amssymb,mathrsfs,color}
\usepackage{microtype,enumitem,float}
\usepackage{tikz}
\usetikzlibrary{arrows.meta,positioning,fit,calc}
\allowdisplaybreaks[2]
\usepackage{hyperref}
\hypersetup{
	colorlinks=true,
	linkcolor=blue,
	anchorcolor=blue,
	citecolor=blue}

\setlist{nosep,leftmargin=*,topsep=2pt,partopsep=0pt}
\AtBeginDocument{%
	\setlength{\abovedisplayskip}{5pt plus 2pt minus 2pt}%
	\setlength{\belowdisplayskip}{5pt plus 2pt minus 2pt}%
	\setlength{\abovedisplayshortskip}{3pt plus 1pt minus 1pt}%
	\setlength{\belowdisplayshortskip}{3pt plus 1pt minus 1pt}%
	\setlength{\textfloatsep}{9pt plus 2pt minus 2pt}%
	\setlength{\intextsep}{7pt plus 2pt minus 2pt}%
	\setlength{\floatsep}{7pt plus 2pt minus 2pt}%
}

\newtheorem{thm}{Theorem}[section]

\newtheorem{lem}[thm]{Lemma}
\newtheorem{prop}[thm]{Proposition}

\theoremstyle{definition}
\newtheorem{de}[thm]{Definition}
\theoremstyle{remark}
\newtheorem{rem}[thm]{Remark}

\numberwithin{equation}{section}

\allowdisplaybreaks[2]

\begin{document}

	\title[A criterion for the well-posedness of MVSDEs] {A criterion for the well-posedness of McKean--Vlasov stochastic differential equations}

	\author{Zhenxin Liu and Ziting Liu}
	\address{Zhenxin Liu: School of Mathematical Sciences,
		Dalian University of Technology, Dalian 116024, P. R. China}
	\email{zxliu@dlut.edu.cn}
	\address{Ziting Liu (Corresponding author): School of Mathematical Sciences,
		Dalian University of Technology, Dalian 116024, P. R. China}
	\email{ZitingLiu@outlook.com}
	

	\date{\today}
	
	

\begin{abstract}
We establish strong existence and pathwise uniqueness for McKean--Vlasov stochastic differential equations with coefficients satisfying a distribution-dependent Lyapunov condition. Under a hybrid Perron--Nagumo condition that permits a non-integrable singularity at the initial time, pathwise uniqueness holds within the class of strong solutions satisfying the corresponding Lyapunov estimate. For existence, we truncate the coefficients on nested bounded domains, construct absorbed local weak solutions, pass to a weak solution via tightness arguments, and then apply a restricted Yamada--Watanabe theorem to obtain a strong solution. Our existence proof, different from the classical truncation-patching method, is interesting in its own right. We also provide an explicit example to which our criterion applies, while none of the Lipschitz, Osgood, or monotonicity conditions is satisfied.
\end{abstract}

\maketitle
	
\section{Introduction}
\setcounter{equation}{0}
McKean--Vlasov stochastic differential equations (MVSDEs), also called mean-field or distribution-dependent SDEs, arose from kinetic models related to the Boltzmann equation and were developed through the work of Kac \cite{FKT}, McKean \cite{MCK}, and Vlasov \cite{VLA}. They describe, among other phenomena, the limiting dynamics of large interacting particle systems and the associated propagation of chaos. We refer to Sznitman \cite{SZN} for a classical account.
	
Let us consider the following general MVSDE:
\begin{equation}\label{1.1}
dX(t)=b\big(t,X(t),\mathcal{L}_{X(t)}\big)dt+\sigma\big(t,X(t),\mathcal{L}_{X(t)}\big)dB(t),
\end{equation}
where $B$ is an $m$-dimensional Brownian motion and $\mathcal{L}_{X(t)}$ denotes the law of $X(t)$. Under the Lipschitz condition, well-posedness is classical. See \cite{SZN} for details. Motivated by Landau-type equations, Wang \cite{LANDAU} constructed strong solutions by iterating in distributions and proved uniqueness under continuity, monotonicity, and moment growth conditions. Ren, Tang, and Wang \cite{RP}, under sufficiently strong noise, proved existence and uniqueness for distribution-path dependent stochastic transport-type equations, arising from stochastic fluid mechanics with forces depending on the history and the environment. For MVSDEs with jumps, Erny \cite{EX} established strong
well-posedness under local Lipschitz conditions together with suitable growth conditions. Galeati, Harang, and Mayorcas \cite{GL} developed well-posedness criteria for MVSDEs driven by additive continuous noise under either Osgood-type, monotonicity, local Lipschitz, or Sobolev-type conditions. Li et al.~\cite{LY} studied the Euler--Maruyama approximation under local Lipschitz conditions in the state variable. Hong, Hu, and Liu \cite{LW} proved strong and weak well-posedness under local monotonicity conditions. Recently, Liu and Ma \cite{Liu-Ma} established existence and uniqueness under distribution-dependent Lyapunov conditions, using a path-space truncation argument. 
	
Classical SDEs, obtained from \eqref{1.1} by dropping the distribution dependence, provide a natural starting point for the present problem. In this setting, none of the Lipschitz, Osgood, or monotonicity conditions is individually necessary for pathwise uniqueness. Ikeda and Watanabe \cite{SDE}, for example, proved pathwise uniqueness for
\begin{equation}\label{1.2}
dX(t)=b\big(t,X(t)\big)dt+\sigma\big(t,X(t)\big)dB(t),
\end{equation}
under an Osgood modulus condition. Motivated by Constantin \cite{CON}, Negrea \cite{NEG} obtained pathwise uniqueness under a combination of Osgood-type and Nagumo-type conditions. Liu and Liu \cite{LZ} replaced the Osgood-type component by a Perron-type condition. In the McKean--Vlasov setting, Bahlali, Mezerdi, and Mezerdi \cite{BK} considered
\begin{equation}\label{1.3}
dX(t)=b\big(t,X(t),\mathcal{L}_{X(t)}\big)dt+\sigma\big(t,X(t)\big)dB(t),
\end{equation}
with bounded coefficients that are Lipschitz in the distribution variable and Osgood continuous in the state variable. Kalinin, Meyer--Brandis, and Proske \cite{KA} further developed multidimensional well-posedness criteria involving partial and local Osgood conditions. These developments suggest seeking a criterion that retains the flexibility of the Perron--Nagumo approach while accommodating both distribution dependence and Lyapunov localization. 
	
In the present paper, we establish the existence of a strong solution to (\ref{1.1}) satisfying an \emph{a priori} Lyapunov estimate, and we prove that pathwise uniqueness holds among all such strong solutions. The coefficients are localized via a distribution-dependent Lyapunov function, while a hybrid Perron--Nagumo increment condition controls the state and distribution variables and permits a non-integrable singularity at the initial time. Our uniqueness criterion extends the results of Liu and Ma \cite{Liu-Ma} by providing a Perron--Nagumo-type uniqueness condition. The existence proof proceeds by truncating the coefficients on nested bounded domains, constructing absorbed weak solutions locally, and passing to the limit via tightness to obtain a weak solution. A restricted Yamada--Watanabe theorem then yields a strong solution satisfying the Lyapunov estimate. 

Classical pathwise truncation-and-patching arguments for SDEs do not carry over directly to the McKean--Vlasov setting. Once the state process is truncated, its law deviates from the distribution parameter appearing in the coefficients, breaking the self-consistency between the process and the distribution variable. In the path-dependent framework, truncation in the path variable is natural. Ren, Tang, and Wang~\cite{RP} established well-posedness for path-dependent McKean--Vlasov equations by truncating the coefficients on the path space. Liu and Ma~\cite{Liu-Ma} adapted this path-space truncation strategy to (\ref{1.1}) under distribution-dependent Lyapunov conditions. Our existence proof proceeds along a different route, which is natural and simple, and relies on neither of these methods. We regard this method as one of the main contributions of our work.

The remainder of the paper is organized as follows. Section~2 collects auxiliary results, Section~3 proves the main theorem, and Section~4 presents an explicit example that satisfies the present criterion while violating the Lipschitz, Osgood, and monotonicity conditions.
	
\section{Preliminaries}
\setcounter{equation}{0}
We first introduce some notation. Let $\|\cdot\|_{\mathbb{R}^d}$ be the Euclidean vector norm. Let $A\in\mathbb{R}^{d\times m}$ be a matrix and define its norm by $\|A\|_{\mathbb{R}^{d\times m}}:=\sqrt{\textnormal{tr}(A^{\top}A)}$, where $\textnormal{tr}$ denotes the trace of a square matrix. Let $\mathcal{P}(\mathbb{R}^d)$ represent the space of all probability measures on $\mathbb{R}^d$ equipped with the weak topology and set
\[
\mathcal{P}_p(\mathbb{R}^d):=\bigg\{\mu\in\mathcal{P}(\mathbb{R}^d):\int_{\mathbb{R}^d}\|x\|^p_{\mathbb{R}^d}\mu(dx)<\infty\bigg\},
\]
where $p\geq1$. Then $\mathcal{P}_p(\mathbb{R}^d)$ is a Polish space under the $L^p$-Wasserstein distance
\[
\mathbb{W}_p(\mu,\nu):=\bigg(\inf\limits_{\gamma\in\Gamma(\mu,\nu)}\int_{\mathbb{R}^d\times\mathbb{R}^d}\|x-y\|^p_{\mathbb{R}^d}\gamma(dx,dy)\bigg)^{\frac{1}{p}},\quad\mu,\nu\in\mathcal{P}_p(\mathbb{R}^d),
\]
where $\Gamma(\mu,\nu)$ stands for the set of all couplings for $\mu$ and $\nu$.
	
Let $\mathcal{C}_T:=C([0,T];\mathbb{R}^d)$, equipped with the uniform norm $\|\xi\|_T:=\sup_{t\in[0,T]}\|\xi(t)\|_{\mathbb{R}^d}$ for $T>0$. We denote by $\mathcal{P}_T$ the space of all probability measures on $\mathcal{C}_T$ equipped with the weak topology. For $p\geq1$, we define the subset $\mathcal{P}_{p,T}$ by
\[
\mathcal{P}_{p,T}:=\bigg\{\mu\in\mathcal{P}_T:\int_{\mathcal{C}_T}\|\xi\|_T^p\mu(d\xi)<\infty\bigg\}.
\]
	
We now consider the following MVSDE:
\begin{equation}\label{2.1}
\begin{cases}
dX(t)=b(t,X(t),\mathcal{L}_{X(t)})dt+\sigma(t,X(t),\mathcal{L}_{X(t)})dB(t),\\
X(0)=X_0.
\end{cases}
\end{equation}
Here $b:[0,\infty)\times\mathbb{R}^d\times\mathcal{P}_2(\mathbb{R}^d)\rightarrow\mathbb{R}^d$ and $\sigma:[0,\infty)\times\mathbb{R}^d\times\mathcal{P}_2(\mathbb{R}^d)\rightarrow\mathbb{R}^{d\times m}$ are measurable maps, $B$ is a standard $m$-dimensional Brownian motion defined on the complete filtered probability space $(\Omega,\mathcal{F},\{\mathcal{F}_t\}_{t\geq0},\mathbb{P})$, and $X_0$ is $\mathcal{F}_0$-measurable.
	
We next define strong and weak solutions, as well as pathwise uniqueness. 
	
\begin{de}\label{de.2.1}(1) Fix a filtered probability space carrying an $\mathbb{R}^m$-valued Brownian motion $B$ and an $\mathcal{F}_0$-measurable random variable $X_0$. An adapted continuous $\mathbb{R}^d$-valued process $X$ is called a \emph{strong solution} of (\ref{2.1}) if
\[
\mathbb{E} \int_{0}^{t}\bigg\{\big\|b(s,X(s),\mathcal{L}_{X(s)})\big\|_{\mathbb{R}^d}+\big\|\sigma(s,X(s),\mathcal{L}_{X(s)})\big\|^2_{\mathbb{R}^{d\times m}}\bigg\}ds<\infty,\,\, t\geq0,
\]
and if, $\mathbb{P}$-almost surely,
\[
X(t)=X_0+\int_{0}^{t}b(s,X(s),\mathcal{L}_{X(s)})ds+\int_{0}^{t}\sigma(s,X(s),\mathcal{L}_{X(s)})dB(s),\,\, t\geq0,
\]
where $\mathcal L_{X(t)}\in\mathcal P_2(\mathbb R^d)$ for every $t\geq0$. \emph{Pathwise uniqueness} holds if any two such strong solutions defined with respect to the same $B$ and the same $X_0$ coincide almost surely.
		
(2) A pair $(\tilde{X}(t),\tilde{B}(t))_{t\geq0}$ is called a \emph{weak solution} of $(\ref{2.1})$ if it solves the equation on some stochastic basis $(\tilde{\Omega},\tilde{\mathcal{F}},\{\tilde{\mathcal{F}}_t\}_{t\geq0},\tilde{\mathbb{P}})$ and $\tilde B$ is an $\mathbb R^m$-valued Brownian motion on that basis.
\end{de}
	
We next collect several auxiliary results used in Section~3. 
	
\begin{de}\label{definition.u}
For a continuous function $F:[0,t_0]\times\mathbb{R}\rightarrow\mathbb{R}$,
\begin{enumerate}
\item A continuous function $f:[0,t_0]\rightarrow\mathbb{R}$ is a {\it lower function} of $F$ if $f(0)=0$ and $D_{\pm}f(t)<F(t,f(t))$ for $t\in(0,t_0)$. Moreover, $D_+f(0)<F(0,f(0))$ and $D_-f(t_0)<F(t_0,f(t_0))$.
\item A continuous function $g:[0,t_0]\rightarrow\mathbb{R}$ is an {\it upper function} of $F$ if $g(0)=0$ and $D_{\pm}g(t)>F(t,g(t))$ for $t\in(0,t_0)$. Moreover, $D_+g(0)>F(0,g(0))$ and $D_-g(t_0)>F(t_0,g(t_0))$.
\end{enumerate}
Here, $D_{+}$ and $D_{-}$ denote the forward and backward derivatives, respectively. The definition presupposes that these one-sided derivatives exist at every $t\in[0,t_0]$.
\end{de}
\begin{de}\label{5.22}
Fix an atomless probability space $(\Omega,\mathcal F,\mathbb P)$ rich enough to support a random variable with any prescribed
law in $\mathcal P_2(\mathbb R^d)$. A function
$u:\mathcal P_2(\mathbb R^d)\to\mathbb R$ is said to be {\it L-differentiable} at $\mu_0\in\mathcal{P}_2(\mathbb{R}^d)$ if there exists a random variable $X_0$ with law $\mu_0$ such that the lifted function $\tilde{u}$ is Fr\'{e}chet differentiable at $X_0$. The lift of the function $\mathcal{P}_2(\mathbb{R}^d)\ni\mu\mapsto u(\mu)$ is the function $\tilde{u}$ defined on the Hilbert space $L^2(\Omega,\mathcal{F},\mathbb{P};\mathbb{R}^d)$ by $\tilde{u}(X):=u(\mathcal{L}_X)$. 
\end{de}
		
\begin{prop}[Joint Chain Rule]\cite[Proposition 5.102]{CAR}\label{lem.chain rule}
For a given $T>0$, let $V:[0,T]\times\mathbb{R}^d\times\mathcal{P}_2(\mathbb{R}^d)\to\mathbb{R}$ be a continuous function such that
\begin{enumerate}
\item[$\rm(1)$] For any $\mu\in\mathcal{P}_2(\mathbb{R}^d)$, the function $[0,T]\times\mathbb{R}^d\ni(t,x)\mapsto V(t,x,\mu)$ is of class $\mathcal{C}^{1,2}$, namely $\partial_tV$, $\partial_xV$ and $\partial^2_{xx}V$ are continuous in $(t,x,\mu)$.	
\item[$\rm(2)$] For any $(t,x)\in[0,T]\times\mathbb{R}^d$, the function $\mathcal{P}_2(\mathbb{R}^d)\ni\mu\mapsto V(t,x,\mu)$ is continuously L-differentiable and, for any $\mu\in\mathcal{P}_2(\mathbb{R}^d)$, we can find a version of the mapping $\mathbb{R}^d\ni v\mapsto\partial_\mu V(t,x,\mu)(v)$ such that the mapping $[0,T]\times\mathbb{R}^d\times\mathcal{P}_2(\mathbb{R}^d)\times\mathbb{R}^d\ni(t,x,\mu,v)\mapsto\partial_\mu V(t,x,\mu)(v)$ is locally bounded and is continuous at any $(t,x,\mu,v)$ such that $v\in\textnormal{supp}(\mu)$.
\item[$\rm(3)$] For the version of $\partial_\mu V$ mentioned above and for any $(t,x,\mu)\in[0,T]\times\mathbb{R}^d\times\mathcal{P}_2(\mathbb{R}^d)$, the mapping $\mathbb{R}^d\ni v\mapsto\partial_\mu V(t,x,\mu)(v)\in\mathbb{R}^d$ is continuously differentiable and its derivative, denoted by $\mathbb{R}^d\ni v\mapsto\partial_v\partial_\mu V(t,x,\mu)(v)\in\mathbb{R}^{d\times d}$, is locally bounded and is jointly continuous in $(t,x,\mu,v)$ at any point $(t,x,\mu,v)$ such that $v\in\textnormal{supp}(\mu)$.
\end{enumerate}
Assume further that, for every compact subset $\mathcal{K}\subset\mathbb{R}^d\times\mathcal{P}_2(\mathbb{R}^d)$, 
\[
\sup_{(t,x,\mu)\in[0,T]\times\mathcal{K}}\Big\{\int_{\mathbb{R}^d}\big\|\partial_\mu V(t,x,\mu)(v)\big\|^2_{\mathbb{R}^d}d\mu(v)+\int_{\mathbb{R}^d}\big\|\partial_v\partial_\mu V(t,x,\mu)(v)\big\|^2_{\mathbb{R}^{d\times d}}d\mu(v)\Big\}<\infty.
\]
Moreover, let $\{X(t)\}_{0\leq t\leq T}$ be an It\^o process of the form 
\[
dX(t)=b(t)dt+\sigma(t)dB(t),\,\,X(0)=X_0\in L^2(\Omega,\mathcal{F},\mathbb{P}),
\]
where $\{B(t)\}_{t\geq0}$ is an $\{\mathcal{F}_t\}_{t\geq0}$-Brownian motion with values in $\mathbb{R}^m$, $\{b(t)\}_{t\geq0}$ and $\{\sigma(t)\}_{t\geq0}$ are $\{\mathcal{F}_t\}_{t\geq0}$-progressively measurable processes with values in $\mathbb{R}^d$ and $\mathbb{R}^{d\times m}$, respectively, and satisfy
\[
\mathbb{E}\Big\{\int_{0}^{T}\big(\|b(t)\|^2_{\mathbb{R}^d}+\|\sigma(t)\|^2_{\mathbb{R}^{d\times m}}\big)dt\Big\}<\infty.
\] 
Let $\{\xi(t)\}_{t\in[0,T]}$ be another $d$-dimensional It\^{o} process on the same filtered probability space $(\Omega,\mathcal{F},\{\mathcal{F}_t\}_{t\geq0},\mathbb{P})$ of the form $d\xi(t)=\eta(t)dt+\gamma(t)dB(t)$, $\xi(0)=\xi_0\in L^2(\Omega,\mathcal{F},\mathbb{P})$, where the $\{\mathcal{F}_t\}_{t\geq0}$-progressively measurable processes $\{\eta(t)\}_{t\in[0,T]}$ and $\{\gamma(t)\}_{t\in[0,T]}$ with values in $\mathbb{R}^d$ and $\mathbb{R}^{d\times m}$, respectively, satisfy
\[
\mathbb{P}\Big\{\int_{0}^{T}\big(\|\eta(t)\|_{\mathbb{R}^d}+\|\gamma(t)\|^2_{\mathbb{R}^{d\times m}}\big)dt<\infty\Big\}=1.
\]
Then, almost surely, for all $t\in[0,T]$,
\begin{equation}\label{Joint chain rule}
\begin{aligned}
V(t,\xi(t),\mathcal{L}_{X(t)})&=V(0,\xi_0,\mathcal{L}_{X_0})+\int_{0}^{t}\partial_xV(r,\xi(r),\mathcal{L}_{X(r)})\cdot(\gamma(r)dB(r))\\
&+\int_{0}^{t}\big(\partial_tV(r,\xi(r),\mathcal{L}_{X(r)})+\partial_xV(r,\xi(r),\mathcal{L}_{X(r)})\cdot\eta(r)\big)dr\\				&+\frac{1}{2}\int_{0}^{t}\textnormal{tr}\big[\partial^2_{xx}V(r,\xi(r),\mathcal{L}_{X(r)})(\gamma\gamma^\top)(r)\big]dr\\
&+\int_{0}^{t}\tilde{\mathbb{E}}\big[\partial_\mu V(r,\xi(r),\mathcal{L}_{X(r)})(\tilde{X}(r))\cdot\tilde{b}(r)\big]dr\\
&+\frac{1}{2}\int_{0}^{t}\tilde{\mathbb{E}}\big[\textnormal{tr}\big(\partial_v\partial_\mu V(r,\xi(r),\mathcal{L}_{X(r)})(\tilde{X}(r))(\tilde{\sigma}\tilde{\sigma}^\top)(r)\big)\big]dr,
\end{aligned}	
\end{equation}
where the process $(\tilde{X}(t),\tilde{b}(t),\tilde{\sigma}(t))_{t\in[0,T]}$ is a copy of the process $(X(t),b(t),\sigma(t))_{t\in[0,T]}$ defined on a copy $(\tilde{\Omega},\tilde{\mathcal{F}},\tilde{\mathbb{P}})$ of $(\Omega,\mathcal{F},\mathbb{P})$.
\end{prop} 
	
We define the operator $\mathcal{L}$ by
\begin{equation}\label{generator*}
\begin{aligned}
(\mathcal{L}V)(t,x,\mu)&:=\partial_tV(t,x,\mu)+\partial_xV(t,x,\mu)\cdot b(t,x,\mu)\\
&\quad+\frac{1}{2}\textnormal{tr}[\partial^2_{xx}V(t,x,\mu)\sigma\sigma^\top(t,x,\mu)]\\
&\quad+\int_{\mathbb{R}^d}\partial_\mu V(t,x,\mu)(y)\cdot b(t,y,\mu)d\mu(y)\\
&\quad+\frac{1}{2}\int_{\mathbb{R}^d}\textnormal{tr}[\partial_v\partial_\mu V(t,x,\mu)(y)\sigma\sigma^\top(t,y,\mu)]d\mu(y).
\end{aligned}
\end{equation}
Moreover, define the integrated Lyapunov functional $\mathcal{V}:[0,T]\times\mathcal{P}_2(\mathbb{R}^d)\rightarrow\mathbb{R}$ by
\begin{equation}\label{functional}
\mathcal{V}(t,\mu):=\int_{\mathbb{R}^d}V(t,x,\mu)\mu(dx).
\end{equation}
For later use, we introduce the notation
\begin{equation}\label{marginal LE}
\mathcal{G}\mathcal{V}(t,z,\mu):=b(t,z,\mu)\cdot\partial_\mu\mathcal{V}(t,\mu)(z)+\frac{1}{2}\textnormal{tr}\big(\sigma\sigma^\top(t,z,\mu)\partial_z\partial_\mu\mathcal{V}(t,\mu)(z)\big).
\end{equation}
	
The Skorokhod representation theorem states that the weak limit of a sequence of probability measures can be realized as the almost sure limit of a sequence of random variables on a common probability space.
	
\begin{prop}[Skorokhod representation theorem]\cite[Theorem 6.7]{CPM}\label{Skorokhod}
Let $E$ be a separable metric space, and let $\{\mu_n\}_{n \geq 1}$ be a sequence of probability measures on $E$ that converges weakly to a probability measure $\mu$. Then there exist a probability space $(\tilde{\Omega}, \tilde{\mathcal{F}}, \tilde{\mathbb{P}})$ and random variables $\tilde{X}_n, \tilde{X} : \tilde{\Omega} \to E$ such that
\begin{enumerate}
\item[$\rm(1)$] $\tilde{\mathcal{L}}_{\tilde{X}_n} = \mu_n$ and $\tilde{\mathcal{L}}_{\tilde{X}} = \mu$ under $\tilde{\mathbb{P}}$.
\item[$\rm(2)$] $\tilde{X}_n \to \tilde{X}$ $\tilde{\mathbb{P}}$-almost surely as $n \to \infty$.
\end{enumerate}
\end{prop}
	
We now recall Vitali's convergence theorem, which is often used to prove $L^p$-convergence.
	
\begin{prop}[Vitali's convergence theorem]\cite[Proposition 1.1]{ISI}\label{lem.2.3}
Let $(E,\|\cdot\|)$ be a separable Banach space and let $p\in[1,\infty)$. Let $\{X_n\}_{n\geq1}$ and $X$ be $E$-valued random variables in $L^p$ such that $X_n$ converges to $X$ in probability or almost surely. Then the following statements are equivalent.
\begin{enumerate}
\item[$\rm(1)$] $\{X_n\}$ converges to $X$ in $L^p$.
\item[$\rm(2)$] $\{\|X_n\|^p\}$ is uniformly integrable.
\item[$\rm(3)$] $\lim\limits_{n \to \infty} \mathbb{E}[\|X_n\|^p] = \mathbb{E}[\|X\|^p]$.
\end{enumerate}
Furthermore, if either of the following two conditions holds, then $\rm (1)$--$\rm (3)$ are satisfied.
\begin{enumerate}
\item[$\rm(4)$] $\sup\limits_{n} \mathbb{E}[\|X_n\|^q] < \infty$ for some $q \in (p, \infty)$.
\item[$\rm(5)$] There exists a random variable $Y \in L^p$ such that $\|X_n\| \le Y$ for all $n$.
\end{enumerate}
\end{prop}
	
\section{Main result}
\setcounter{equation}{0}
We first establish several auxiliary lemmas and then prove the existence and uniqueness result under the combined Perron--Nagumo and distribution-dependent Lyapunov conditions.
\begin{lem}[Concatenation of absorbed one-step weak solutions]\label{lem:splicing}
Let $D\subset\mathbb R^d$ be open and $\overline{D}$ compact, let $0=s_0<s_1<\cdots<s_N=T$, and put $\varphi(t):=s_i$ for $t\in[s_i,s_{i+1})$.  Let $b$ and $\sigma$ be jointly Borel measurable. For a continuous path $\xi$ on $[0,T]$, write
\[
\tau_D(\xi):=\inf\{t\in[0,T]:\xi(t)\notin D\},
\]
where the infimum of the empty set is $\infty$. Suppose that, for every $i\in\{0,\cdots,N-1\}$, $y\in\overline D$, and every $\mu\in\mathcal P_2(\mathbb R^d)$ supported by $\overline D$, there is a weak solution $(Y^i,B^i)$ on $[s_i,s_{i+1}]$ of the absorbed frozen-law equation
\begin{equation}\label{one-step-fixed-law-sde}
Y^i(t)=y+\int_{s_i}^t\mathbf 1_{\{r<\tau^i_D(Y^i)\}}b\bigl(r,Y^i(r),\mu\bigr)dr+\int_{s_i}^t\mathbf 1_{\{r<\tau^i_D(Y^i)\}}\sigma\bigl(r,Y^i(r),\mu\bigr)\,dB^i(r),
\,\,t\in[s_i,s_{i+1}].
\end{equation}
The solution is absorbed after $\tau^i_D(Y^i):=\inf\{t\in[s_i,s_{i+1}]:Y^i(t)\notin D\}$ and has a finite second moment.
Then the joint laws of these stopped solutions admit a universally measurable selection in $(y,\mu)$.  Consequently, every initial law $\nu_0\in\mathcal P_2(\mathbb R^d)$ supported by $\overline D$ admits a concatenation $(X,B)$ on a stochastic basis $(\Omega,\mathcal{F},\{\mathcal{F}_t\},\mathbb{P})$ such that $B$ is a Brownian motion, $\mathcal L_{X(0)}=\nu_0$, and, with $\tau_D:=\inf\{t\in[0,T]:X(t)\notin D\}$, the process $X$ is absorbed after $\tau_D$ and satisfies
\begin{equation}\label{spliced-euler-equation}
X(t)=X(0)+\int_0^t\mathbf1_{\{r<\tau_D\}}b\bigl(r,X(r),\mathcal L_{X(\varphi(r))}\bigr)dr+\int_0^t\mathbf1_{\{r<\tau_D\}}\sigma\bigl(r,X(r),\mathcal L_{X(\varphi(r))}\bigr)dB(r).
\end{equation}
\end{lem}
	
\begin{proof}
For a fixed $i$, consider the set of triples $(y,\mu,Q)$ for which $Q$ is the joint law of a pair $(Y^i,B^i)$ solving \eqref{one-step-fixed-law-sde}. The martingale identities tested at rational times against a countable determining class show that this set is analytic.  Its projection onto $(y,\mu)$ is the whole parameter set by the assumed one-step weak existence. The Jankov--von Neumann selection theorem therefore yields a universally measurable kernel $Q_i(y,\mu)$; see the standard martingale-problem selection argument in \cite[Chapter~12]{STR} and the descriptive-set-theoretic selection theorem in \cite[Chapter~18]{Kechris}. The stopping map on continuous paths is Borel, so this is also a universally measurable kernel for the stopped pairs.
		
Set $\mu_0:=\nu_0$.  Given $\mu_i$, use $Q_i(y,\mu_i)$ and define
\[
\mu_{i+1}(A):=\int_{\overline D}Q_i(y,\mu_i)\bigl(Y(s_{i+1})\in A\bigr)\,\mu_i(dy).
\]
Absorption gives $\mu_{i+1}(\overline D)=1$, and the moment assumption gives $\mu_{i+1}\in\mathcal P_2(\mathbb R^d)$. The Ionescu--Tulcea theorem \cite{IonescuTulcea}, applied to the universally measurable kernels $Q_0,\ldots,Q_{N-1}$ after the usual completion, produces the concatenated state process and the concatenated Brownian increments. Conditionally on the past at $s_i$, the increment on $[s_i,s_{i+1}]$ is a Wiener increment independent of that past; successive conditioning therefore proves that the concatenated noise is a Brownian motion. Since $\mathcal L_{X(\varphi(r))}=\mu_i$ on the $i$-th interval, summing the stopped one-step equations gives
\eqref{spliced-euler-equation}.
\end{proof}
The following lemma retains the notation of Lemma~\ref{lem:splicing}.
\begin{lem}[Absorbed Euler limit]\label{lem:absorbed-euler-limit}
Let $D\subset\mathbb R^d$ be open with compact closure. Suppose that $b$ and $\sigma$ are bounded on $[0,T]\times\overline D\times\mathcal{M}$ and continuous in $(x,\mu)$ there, where $\mathcal{M}$ is a compact set of probability measures supported by $\overline D$. Let the mesh sizes of a sequence of partitions converge to zero. Let $(X^n,B^n)$ be the absorbed Euler concatenations of Lemma~{\rm \ref{lem:splicing}}, with $\mathcal L_{X^n(0)}=\nu_0$ and with all their marginal laws in $\mathcal{M}$. Then there exist a subsequence, a stochastic basis $(\Omega,\mathcal{F},\{\mathcal{F}_t\},\mathbb{P})$, an $m$-dimensional Brownian motion $B$, a continuous adapted process $X$, and a stopping time $\theta$ with values in $[0,T]\cup\{\infty\}$ such that
\begin{equation}\label{absorbed-limit-sde}
X(t)=X(0)+\int_0^t\mathbf1_{\{r<\theta\}}b\bigl(r,X(r),\mathcal{L}_{X(r)}\bigr)dr+\int_0^t\mathbf1_{\{r<\theta\}}\sigma\bigl(r,X(r),\mathcal{L}_{X(r)}\bigr)dB(r),\,\,0\leq t\leq T,
\end{equation}
and
\begin{equation}\label{delayed-stopped-path}
X(t)=X(t\wedge\theta),\,\,
X(\theta)\in\partial D\quad\hbox{on }\{\theta\leq T\}.
\end{equation}
In general, $\theta$ need not be the first exit time of $X$ from $D$.
\end{lem}
	
\begin{proof}
For the $n$-th partition $0=s_0^n<\cdots<s_{N_n}^n=T$ with $\lim\limits_{n\rightarrow\infty}\max\limits_{0\leq i\leq N_n-1}|s^n_{i+1}-s^n_i|=0$, let
$\varphi_n(t):=s_i^n$ for $t\in[s_i^n,s_{i+1}^n)$, $i\in\{0,\dots,N_n-1\}$ and $\varphi_n(T)=T$. Put
\[
\tau^n:=\inf\{t\in[0,T]:X^n(t)\notin D\},\,\,
A^n(t):=\mathbf1_{\{t<\tau^n\}},
\]
with the convention $\inf\varnothing=\infty$.

Define 
\[
\widehat\theta^n:=
\begin{cases}
1+\tau^n,&\tau^n\leq T,\\
T+2,&\tau^n=\infty,
\end{cases}
\,\,\widehat A^n(s):=\mathbf1_{\{s<\widehat\theta^n\}},\,\, 0\leq s\leq T+2.
\]
We have
\begin{equation}\label{extended-activity-restriction}
\widehat A^n(t+1)=A^n(t),\,\, 0\leq t\leq T.
\end{equation}
For $q\in[0,T+2]$, write $a_q(s):=\mathbf1_{\{s<q\}}$ and set
\[
\mathscr A:=\{a_q:q\in[1,T+1]\}\cup\{a_{T+2}\}
\subset D([0,T+2];\{0,1\}),
\]
where $D([0,T+2];\{0,1\})$ denotes the space of
$\{0,1\}$-valued c\`adl\`ag functions on $[0,T+2]$, equipped
with the Skorokhod $J_1$ topology (see
\cite[Section~12]{CPM}). $\mathscr{A}$ is a compact set in the $J_1$ topology. Indeed, if $q_n\to q$ in $[1,T+1]$, the piecewise linear increasing homeomorphism which maps $0$ to $0$, $q$ to $q_n$, and $T+2$ to $T+2$ shows that $a_{q_n}\to a_q$ in $J_1$. The image of $[1,T+1]$ is therefore compact, and $a_{T+2}$ is an additional singleton. Since $\widehat A^n\in\mathscr A$ almost surely, the indicator processes $\widehat{A}^n$ are tight. For $1\leq s\leq T+1$, put
\[
\widehat X^n(s):=X^n(s-1),\,\,
\widehat B^n(s):=B^n(s-1),\,\,
\widehat\varphi_n(s):=1+\varphi_n(s-1).
\]
Thus $\widehat\varphi_n(s)$ is defined on the whole interval $[1,T+1]$. With respect to the shifted filtration $\widehat{\mathcal F}^{\,n}_s:=\mathcal F^n_{s-1}$, $\widehat B^n$ is a Brownian motion on $[1,T+1]$, and $\widehat A^n(s)$ is $\widehat{\mathcal F}^{\,n}_s$-measurable. Boundedness gives the usual uniform moment and modulus estimates for $(\widehat X^n,\widehat B^n)$. Hence the laws of $(\widehat X^n,\widehat B^n,\widehat A^n)$ are tight on
\[
C([1,T+1];\mathbb R^d)\times C([1,T+1];\mathbb R^m)\times D([0,T+2];\{0,1\}),
\]
where the last factor is equipped with the $J_1$ topology. By Prokhorov's theorem and the Skorokhod representation theorem, after passing to a subsequence which is not relabelled, there exist a probability space \((\widetilde{\Omega},\widetilde{\mathcal F},
\widetilde{\mathbb P})\) and random variables $(\widetilde X^n,\widetilde B^n,\widetilde A^n)$ and $ (\widetilde X,\widetilde B,\widetilde A)$ taking values in
\[
C([1,T+1];\mathbb R^d)\times C([1,T+1];\mathbb R^m)\times D([0,T+2];\{0,1\}),
\]
such that $\mathcal L_{\widetilde{\mathbb P}} (\widetilde X^n,\widetilde B^n,\widetilde A^n)=
\mathcal L (\widehat X^n,\widehat B^n,\widehat A^n)$ for every $n$, and
\[
(\widetilde X^n,\widetilde B^n,\widetilde A^n)
\rightarrow
(\widetilde X,\widetilde B,\widetilde A),\,\, \widetilde{\mathbb P}\text{-a.s.}
\]
The convergence is uniform in the first two coordinates and is in the \(J_1\) topology in the third coordinate. 

Since \(\mathscr A\) is closed, $\widetilde A=a_{\widehat\theta}$ for a unique \(\widehat\theta\in[1,T+1]\cup\{T+2\}\). Define
\[
\theta:=
\begin{cases}
\widehat\theta-1,&\widehat\theta\leq T+1,\\
\infty,&\widehat\theta=T+2.
\end{cases}
\]
Moreover, for every \(n\), there is a unique
\(\widetilde q^n\in[1,T+1]\cup\{T+2\}\) such that \(\widetilde A^n=a_{\widetilde q^n}\). Put
\[
\widetilde\tau^n:=
\begin{cases}
\widetilde q^n-1,&\widetilde q^n\leq T+1,\\
\infty,&\widetilde q^n=T+2.
\end{cases}
\]
Then
\[
\widetilde A^n(t+1)=\mathbf1_{\{t<\widetilde\tau^n\}},
\,\,\widetilde A(t+1)=\mathbf1_{\{t<\theta\}},\,\, 0\leq t\leq T.
\]
Since \(q\mapsto a_q\) is a homeomorphism from \([1,T+1]\) onto its image and \(a_{T+2}\) is isolated from that image, we have
\[
\widetilde q^n\rightarrow\widehat\theta,
\,\,\int_0^T\left|\mathbf1_{\{t<\widetilde\tau^n\}}-\mathbf1_{\{t<\theta\}}\right|dt\rightarrow0,
\,\,\widetilde{\mathbb P}\text{-a.s.}
\]
Indeed, if $\widehat\theta\leq T+1$, then eventually
$\widetilde q^n\leq T+1$ and
$\widetilde\tau^n=\widetilde q^n-1\to
\widehat\theta-1=\theta$. Hence
\[
\int_0^T
\left|
\mathbf1_{\{t<\widetilde\tau^n\}}
-\mathbf1_{\{t<\theta\}}
\right|dt
=
|\widetilde\tau^n-\theta|
\rightarrow0.
\]
If $\widehat\theta=T+2$, the isolation of $a_{T+2}$ implies that $\widetilde q^n=T+2$ for all sufficiently large $n$, so both indicators are equal to one on $[0,T]$. Let $\mu^n(s):=\mathcal L_{\widetilde{\mathbb P}}(\widetilde X^n(s)),\,\,\nu^n(s):=\mu^n(\widehat\varphi_n(s)),\,\,1\leq s\leq T+1$. We have
\[
\sup_{1\leq s\leq T+1}\mathbb W_2\bigl(\mu^n(s),\nu^n(s)\bigr)\rightarrow0.
\]
Furthermore, the almost sure uniform convergence of the state coordinates and compact support imply
\[
\sup_{1\leq s\leq T+1}\mathbb W_2\bigl(\mu^n(s),\mu(s)\bigr)\rightarrow0,\,\,
\mu(s):=\mathcal L_{\widetilde{\mathbb P}}(\widetilde X(s)).
\]
Since $\mathcal M$ is compact and hence closed in the $\mathbb W_2$-topology, and $\mu^n(s)\in\mathcal M$ for every $n$ and $s\in[1,1+T]$, it follows that $\mu(s)\in\mathcal M$.

We define, for \(1\leq s\leq T+1\),
\[
\begin{aligned}
	\widetilde M^n(s)&:=\widetilde X^n(s)-\widetilde X^n(1)-\int_1^s\widetilde A^n(r)b\bigl(r-1,\widetilde X^n(r),\nu^n(r)\bigr)dr,\\
	Q_{ij}^n(s)
	&:=
	\int_1^s
	\widetilde A^n(r)
	(\sigma\sigma^\top)_{ij}
	\bigl(r-1,\widetilde X^n(r),\nu^n(r)\bigr)\,dr,\\
	C_{iq}^n(s)
	&:=
	\int_1^s
	\widetilde A^n(r)
	\sigma_{iq}
	\bigl(r-1,\widetilde X^n(r),\nu^n(r)\bigr)\,dr.
\end{aligned}
\]
Since the canonical martingale identities are determined by the joint law, each of the processes
\[
\widetilde M_i^n,\,\,\widetilde B_p^n,\,\,
\widetilde M_i^n\widetilde M_j^n-Q_{ij}^n,
\,\,
\widetilde M_i^n\widetilde B_q^n-C_{iq}^n,
\,\,
\widetilde B_p^n\widetilde B_q^n
-\delta_{pq}(\cdot-1),
\]
is a martingale in the usual augmentation of the canonical filtration generated by $(\widetilde X^n,\widetilde B^n,\widetilde A^n)$.

Choose a countable dense set
$\mathbb T_0\subset(1,T+1]$ such that $\widetilde{\mathbb P}(\widehat\theta=t)=0$ for $t\in\mathbb T_0$.
The uniform convergence of
$(\widetilde X^n,\widetilde B^n)$, the
$L^1$-convergence of $\widetilde{A}^n$, and the boundedness and continuity of the coefficients imply that $\widetilde M^n$, $Q^n$, and $C^n$ converge to their corresponding limiting processes at all times in $\mathbb T_0$, almost surely. Moreover, $\widetilde M^n$ is uniformly bounded, while $\widetilde B^n$ has uniformly bounded moments of every finite order. Hence the random variables appearing in the martingale identities are uniformly integrable. Testing these identities against bounded continuous cylinder functions of the past and passing to the limit gives the corresponding limiting martingale identities on $\mathbb T_0$. A monotone-class argument and right-continuity extend them to the usual augmentation of the
canonical filtration of $(\widetilde X,\widetilde B,\widetilde A)$. In particular, each $\widetilde B_p$ is a continuous martingale and $\widetilde B_p\widetilde B_q-\delta_{pq}(\cdot-1)$ is a martingale. Therefore, \(\widetilde B\) is a Brownian motion with respect to this filtration and
\[
\widetilde M(s):=\widetilde X(s)-\widetilde X(1)-\int_1^s\mathbf1_{\{r<\widehat\theta\}}
b\bigl(r-1,\widetilde X(r),\mu(r)\bigr)\,dr
\]
is a continuous local martingale satisfying
\[
\begin{aligned}
\langle\widetilde M\rangle_s&=\int_1^s\mathbf1_{\{r<\widehat\theta\}}\sigma\sigma^\top\bigl(r-1,\widetilde X(r),\mu(r)\bigr)\,dr,\\
\langle\widetilde M,\widetilde B\rangle_s
&=\int_1^s\mathbf1_{\{r<\widehat\theta\}}
\sigma\bigl(r-1,\widetilde X(r),\mu(r)\bigr)\,dr.
\end{aligned}
\]
Put
\[
H(r):=\mathbf1_{\{r<\widehat\theta\}}\sigma\bigl(r-1,\widetilde X(r),\mu(r)\bigr)
\]
and
\[
N(s):=\widetilde M(s)-\int_1^sH(r)\,d\widetilde B(r).
\]
The preceding quadratic-variation and covariation identities imply
that \(N\) is a continuous local martingale with
\(\langle N\rangle\equiv0\). Hence \(N\equiv0\), and therefore
\[
\widetilde M(s)=\int_1^s\mathbf1_{\{r<\widehat\theta\}}\sigma\bigl(r-1,\widetilde X(r),\mu(r)\bigr)
\,d\widetilde B(r).
\]

Now let $\widehat{\mathcal F}_s$ be the usual augmentation of the canonical filtration generated by $(\widetilde X,\widetilde B,\widetilde A)$ and define
\[
X(t):=\widetilde X(t+1),\,\,
B(t):=\widetilde B(t+1)-\widetilde B(1),\,\,
\mathcal F_t:=\widehat{\mathcal F}_{t+1}.
\]
Then $B$ is a Brownian motion with respect to $\{\mathcal F_t\}$ and $\{\theta\leq t\}=\{\widetilde A(t+1)=0\}\in\mathcal F_t$. Changing variables in the preceding martingale representation proves
\eqref{absorbed-limit-sde}. The convergence also gives $\mathcal L_{X(0)}=\nu_0$.
		
Finally, for every $n\geq1$, we have
\[
\widetilde X^n(t+1)=\widetilde X^n\bigl(1+t\wedge\widetilde\tau^n\bigr),
\,\,0\leq t\leq T,\,\,\widetilde{\mathbb{P}}\text{-a.s.}
\]
and
\[
\widetilde X^n(1+\widetilde\tau^n)\in\partial D
\,\,\text{on }\{\widetilde\tau^n\leq T\}.
\]
We note that $\widetilde X^n\to\widetilde X$ uniformly almost surely on $[1,T+1]$ and $\widetilde\tau^n\to\theta$ almost surely on $\{\theta\leq T\}$. For a fixed sample path and $\delta>0$, define
\[
\omega_{\widetilde X}(\delta)
:=
\sup_{\substack{r,s\in[1,T+1]\\ |r-s|\leq\delta}}
\|\widetilde X(r)-\widetilde X(s)\|_{\mathbb{R}^d}.
\]
Since $\widetilde X$ is continuous on $[1,T+1]$, we have
$\omega_{\widetilde X}(\delta)\to0$ as $\delta\downarrow0$.
Consequently,
\[
    \sup_{0\leq t\leq T}
	\big\|
	\widetilde X(t+1)-\widetilde X(1+t\wedge\theta)
	\big\|_{\mathbb{R}^d}\leq
	2\sup\limits_{t\in[1,1+T]}\|\widetilde X^n(t)-\widetilde X(t)\|_{\mathbb{R}^d}
	+
	\omega_{\widetilde X}
	\bigl(|\widetilde\tau^n-\theta|\bigr)
	\rightarrow0
\]
almost surely on $\{\theta\leq T\}$. On $\{\theta=\infty\}$, the identity is immediate. Hence
\[
X(t)=X(t\wedge\theta),
\,\,0\leq t\leq T,\,\,\widetilde{\mathbb{P}}\text{-a.s.}
\]

Moreover, on $\{\theta\leq T\}$, we have
\[
	\big\|
	\widetilde X^n(1+\widetilde\tau^n)-
	\widetilde X(1+\theta)
	\big\|_{\mathbb{R}^d}\leq
	\sup\limits_{t\in[1,T+1]}\|\widetilde X^n(t)-\widetilde X(t)\|_{\mathbb{R}^d}+
	\omega_{\widetilde X}
	\bigl(|\widetilde\tau^n-\theta|\bigr)
	\rightarrow0.
\]
Since
$\widetilde X^n(1+\widetilde\tau^n)\in\partial D$
eventually and $\partial D$ is closed, we obtain
\[
X(\theta)=\widetilde X(1+\theta)\in\partial D
\,\,\text{on }\{\theta\leq T\}.
\]
This proves \eqref{delayed-stopped-path}.
\end{proof}

The following lemma is a restricted version of the Yamada--Watanabe principle in the compatibility framework of Kurtz~\cite[Theorem~1.5]{Kurtz}, specialized to the Lyapunov-admissible class considered here.
\begin{lem}[Restricted Yamada--Watanabe theorem]
\label{lem:yw-lyapunov-class}
Fix $\mu_0\in\mathcal P_2(\mathbb R^d)$ and define $\mathfrak K_{\mathcal V}\subset\mathcal P(\mathcal C_T)$ by
\[
\mathfrak K_{\mathcal V}
:=\big\{\nu:\nu\circ e_0^{-1}=\mu_0,\,\,
\nu\circ e_t^{-1}\in\mathcal P_2(\mathbb R^d)\text{ for every }t\in[0,T],\,\,
\sup_{t\in[0,T]}\mathcal V(t,\nu\circ e_t^{-1})<\infty\big\},
\]
where $e_t(w):=w(t)$ for $w\in\mathcal{C}_T$.
		
For any weak solution $(X,B)$ of \eqref{2.1} in the sense of Definition~$\rm\ref{de.2.1}$, defined on a filtered probability space $(\Omega,\mathcal F,\{\mathcal F_t\},\mathbb P)$, set
\[
Y:=\bigl(X(0),B|_{[0,T]}\bigr),\,\,Y_t:=\bigl(X(0),B|_{[0,t]}\bigr),\,\,t\in[0,T].
\]
Such a weak solution $(X,B)$ is called compatible if $X$ is continuous and adapted and, for every $t\in[0,T]$ and every bounded Borel function
\[
h:C([0,t];\mathbb R^d)\rightarrow\mathbb R,
\]
one has
\[
\mathbb E\left[
h\bigl(X|_{[0,t]}\bigr)|\sigma(Y)\right]
=
\mathbb E\left[h\bigl(X|_{[0,t]}\bigr)|\sigma(Y_t)\right]\,\,\text{a.s.}
\]
A compatible weak solution $(X,B)$ is called $\mathfrak K_{\mathcal V}$-admissible if $\mathcal L_X\in\mathfrak K_{\mathcal V}$. Assume that
\begin{enumerate}
\item[$\rm(1)$] There exists a compatible $\mathfrak K_{\mathcal V}$-admissible weak solution of \eqref{2.1}.
\item[$\rm(2)$] Pathwise uniqueness in the sense of Definition~$\rm\ref{de.2.1}$ holds within the class of $\mathfrak K_{\mathcal V}$-admissible strong solutions.
\end{enumerate}
		
Then there exists a Borel measurable non-anticipative map
\[
\Phi:\mathbb R^d\times C([0,T];\mathbb R^m)\rightarrow C([0,T];\mathbb R^d)
\]
such that, for every $t\in[0,T]$, there is a Borel map
\[
\Phi_t:\mathbb R^d\times C([0,t];\mathbb R^m)\rightarrow C([0,t];\mathbb R^d)
\]
satisfying
\[
\Phi(x,w)|_{[0,t]}=\Phi_t\bigl(x,w|_{[0,t]}\bigr).
\]
		
		Moreover, on every filtered probability space carrying an
		$\mathbb R^m$-valued Brownian motion $B$ and an
		$\mathcal F_0$-measurable random variable $X_0$ with
		$\mathcal L_{X_0}=\mu_0$, the process
		$X=\Phi(X_0,B)$ is the unique $\mathfrak K_{\mathcal V}$-admissible strong solution of \eqref{2.1} in the sense of
		Definition~$\rm\ref{de.2.1}$.
	\end{lem}
	\begin{proof}
		Let $(X,B)$ be a compatible
		$\mathfrak K_{\mathcal{V}}$-admissible weak solution. Since the path spaces are Polish, there exists a regular conditional
		distribution
		$\Gamma(y,\cdot):=\mathbb P\bigl(X\in\cdot|Y=y\bigr)$. On the conditional-product space with probability measure
		\[
		\overline{\mathbb P}(dx_1,dx_2,dy)
		:=
		\Gamma(y,dx_1)\Gamma(y,dx_2)\mathcal L_{Y}(dy),
		\]
		let $X_1$ and $X_2$ be the first two coordinate processes of the conditional-product space. By the argument of Kurtz~\cite[Lemmas~2.11--2.12]{Kurtz},
		$(X_1,X_2)$ is jointly compatible with $Y$. For every $t\in[0,T]$, the future Brownian increments are independent of the augmentation of $\sigma\bigl(X_1(r),X_2(r),B(r):0\leq r\leq t\bigr)$. Thus $B$ is a Brownian motion with respect to the filtration generated by $(X_1,X_2,B)$. Moreover, we have
		$\mathcal L_{\overline{\mathbb P}}(X_1,X_2,Y)(dx_1,dx_2,dy)=
		\Gamma(y,dx_1)\Gamma(y,dx_2)\mathcal L_Y(dy)$. For $i=1,2$, the $(X_i,Y)$-marginal yields
		\[
		\mathcal L_{\overline{\mathbb P}}(X_i,Y)
		=\mathcal L_{\mathbb P}(X,Y).
		\]
		Hence each $(X_i,B)$ is a compatible
		$\mathfrak K_{\mathcal V}$-admissible weak solution of \eqref{2.1}. Since $X_1$ and $X_2$ are defined on the same filtered probability space, are adapted to the common filtration, are driven by the same Brownian motion $B$, and satisfy
		\[
		X_1(0)=X_2(0)\,\,\overline{\mathbb P}\text{-a.s.},
		\]
		each $X_i$ is also a strong solution in the sense of Definition~\ref{de.2.1}. Therefore, pathwise uniqueness gives
		\[
		\overline{\mathbb{P}}\big(X_1=X_2\text{ in }\mathcal{C}_T\big)=1.
		\]
		Therefore, for $\mathcal L_{Y}$-almost every $y$, the probability
		measure $\Gamma(y,\cdot)$ is a Dirac measure. Hence there exists a
		Borel measurable map $\Phi$ such that
		\[
		\Gamma(y,\cdot)=\delta_{\Phi(y)}
		\,\,\text{for }\mathcal L_{Y}\text{-a.e. }y.
		\]
		Consequently,
		\[
		X=\Phi(Y)\,\,\mathbb P\text{-a.s.}
		\]
		
		Since $X$ is compatible with $Y$,
		Kurtz~\cite[Proposition~2.13]{Kurtz} implies that, for every
		$t\in[0,T]$,
		\[
		\sigma\bigl(X|_{[0,t]}\bigr)
		\subset\sigma(Y_t)
		\]
		up to completion. Hence, by the Doob--Dynkin lemma, for every
		rational $t\in[0,T]$ there exists a Borel map
		\[
		\Phi_t:
		\mathbb R^d\times C([0,t];\mathbb R^m)
		\rightarrow C([0,t];\mathbb R^d)
		\]
		such that
		\[
		\Phi(y)|_{[0,t]}=\Phi_t(\pi_t y)
		\,\,\text{for }\mathcal L_Y\text{-a.e. }y=(x,w),
		\]
		where
		\[
		\pi_t y:=\bigl(x,w|_{[0,t]}\bigr).
		\]
		
		By choosing these versions for rational times and modifying $\Phi$ on an $\mathcal L_Y$-null set, we may use the continuity of the paths to extend this relation to every $t\in[0,T]$. Thus $\Phi$ admits a Borel non-anticipative version. Recalling that $Y=(X(0),B)$, we obtain
		\[
		X=\Phi(X(0),B)\,\,\mathbb P\text{-a.s.}
		\]
		
        Now let $(\Omega',\mathcal F',\{\mathcal F'_t\}_{t\in[0,T]},\mathbb P')$
        be an arbitrary filtered probability space carrying an $\mathbb R^m$-valued Brownian motion $B'$ and an
        $\mathcal F'_0$-measurable random variable $X'_0$ satisfying $\mathcal L_{\mathbb P'}(X'_0)=\mu_0$. Since $B'$ is a Brownian motion with respect to
        $\{\mathcal F'_t\}_{t\in[0,T]}$ and $X'_0$ is $\mathcal F'_0$-measurable, $X'_0$ is independent of $B'$. The same argument shows that $X(0)$ is independent of $B$. Consequently,
        \[
        \mathcal L_{\mathbb P'}(X'_0,B')=\mu_0\otimes\mathsf W_m= \mathcal L_{\mathbb P}(X(0),B),
        \]
        where $\mathsf W_m$ denotes the $m$-dimensional Wiener measure.
        
        Define $X':=\Phi(X'_0,B')$. Since
        $X=\Phi(X(0),B)\,\,\mathbb P\text{-a.s.}$, the equality of the input laws yields
        $\mathcal L_{\mathbb P'}(X'_0,B',X')=
        \mathcal L_{\mathbb P}(X(0),B,X)$.
        The canonical martingale and covariation identities
        characterizing \eqref{2.1}, being determined by the joint law,
        therefore transfer from $(X(0),B,X)$ to $(X'_0,B',X')$.
        Hence $X'$ satisfies \eqref{2.1} with initial value $X'_0$ and
        driving Brownian motion $B'$.
        
        Moreover, since $\Phi$ is non-anticipative, $X'$ is
        $\{\mathcal F'_t\}$-adapted, and since $\Phi$ takes values in
        $C([0,T];\mathbb R^d)$, the process $X'$ has continuous paths.
        Finally,
        \[
        \mathcal L_{\mathbb P'}(X')
        =
        \mathcal L_{\mathbb P}(X)
        \in\mathfrak K_{\mathcal V}.
        \]
        It follows that $X'$ is a
        $\mathfrak K_{\mathcal V}$-admissible strong solution of
        \eqref{2.1} in the sense of Definition~\ref{de.2.1}. 
        
        Let $\widehat X$ be any other $\mathfrak K_{\mathcal V}$-admissible strong solution of \eqref{2.1} on the same filtered probability space, driven by the same Brownian motion $B'$ and starting from the same initial value $X'_0$. Then $\widehat X$ and $X'=\Phi(X'_0,B')$ are two strong solutions in the sense of Definition~\ref{de.2.1}, driven by the same Brownian motion and having the same initial value. Pathwise uniqueness gives
        \[
        \mathbb P'\left(\widehat X(t)=\Phi(X'_0,B')(t)\text{ for every }t\in[0,T]\right)=1.
        \]
        
        Thus $\Phi(X'_0,B')$ is the unique $\mathfrak K_{\mathcal V}$-admissible strong solution of \eqref{2.1} on the prescribed stochastic basis, up to indistinguishability.
        \end{proof}
	
    	We then turn to the proof of the main result. For a fixed constant $T>0$, we impose the following assumptions: 
	    \begin{enumerate}
	    \item[$\rm(H_1)$](Local boundedness and continuity). There exists a continuous coercive function $V_0:\mathbb{R}^d\rightarrow[0,\infty)$, with $V_0(0)=0$ and $V_0(x)\rightarrow\infty$ as $\|x\|_{\mathbb{R}^d}\rightarrow\infty$. Let $0<R_1<R_2<\cdots\uparrow\infty$ be a sequence such that, with $D_k:=\{V_0<R_k\}$, one has $\overline D_k\subset D_{k+1}$. The functions $b$ and $\sigma$ are bounded on $[0,T]\times \overline{D}_k\times\{\mu:\int_{\mathbb{R}^d}V_0(x)\mu(dx)\leq M\}$ for every $k\geq1$ and $M>0$, and are Borel in time and continuous in $(x,\mu)$ on this set.
		\item[$\rm(H_2)$](Coercivity and growth). There exists a nonnegative function $V$ satisfying the joint chain rule such that for all $t\in[0,T]$ and $\mu\in\mathcal{P}_2(\mathbb{R}^d)$ we have
		\begin{equation}\label{assumption.2}
			\mathcal{V}(t,\mu)=\int_{\mathbb{R}^d}V(t,x,\mu)\mu(dx)\geq\int_{\mathbb{R}^d}V_0(x)\mu(dx).
		\end{equation}
		Moreover, for some constant $K>0$,
		\begin{equation}\label{assumption.2.1}
			\mathcal{V}(0,\mu)\leq K(1+\int_{\mathbb{R}^d}V_0(x)\mu(dx)).
		\end{equation}
		There exist constants $l>1$ and $C>0$ such that for every $t\in[0,T]$ and $\mu\in\mathcal{P}_2(\mathbb{R}^d)$,
		\begin{equation}\label{assumption.3}
			\int_{\mathbb{R}^d}\big(\|b(t,x,\mu)\|_{\mathbb{R}^d}^{2l}+\|\sigma(t,x,\mu)\|_{\mathbb{R}^{d\times m}}^{2l}\big)\mu(dx)\leq C\big(1+\mathcal{V}(t,\mu)\big).
		\end{equation}
		
		\item[$\rm(H_3)$](Localized integrated Lyapunov condition). There exist $\rho,\eta\in L^1([0,T])$ such that for every $k\geq1$, $t\in[0,T]$, and $\mu\in\mathcal{P}_2(\mathbb{R}^d)$,
		\begin{equation}\label{assumption.5}
			\partial_t\mathcal{V}(t,\mu)+\int_{D_k}\mathcal{G}\mathcal{V}(t,z,\mu)\mu(dz)\leq \rho(t)\mathcal{V}(t,\mu)+\eta(t).
		\end{equation}
		\item[$\rm(H_4)$](Perron--Nagumo-type conditions). There exist constants $\alpha_1\in[0,\infty)$ and $\alpha_2\in[0,\frac{1}{4(T\vee1)})$, and a continuous nondecreasing function $u:(0,T]\to[0,\infty)$ such that $u$ is differentiable, $\lim\limits_{t\downarrow0^+}u(t)=0$, and $\liminf\limits_{t\downarrow0}u'(t)=\infty$. Moreover, there exists a continuous function $\omega:[0,T]\times[0,\infty)\to[0,\infty)$ that is nondecreasing and concave in its second variable. The quotient $\alpha\frac{\omega(t,x)}{u(t)}$ admits a continuous extension $F(t,x)$ to
		$[0,T]\times[0,\infty)$,  where $\alpha:=\frac{4\alpha_1(T\vee1)u(T)}{1-4\alpha_2(T\vee1)}$. And $x\equiv0$ is the unique solution of $x'(t)=F(t,x(t))$, $x(0)=0$. In addition, if we further extend $F(t,x)$ to $[0,T]\times\mathbb R$ by setting
		$F(t,x):=0$ for $x<0$, $F$ admits a $C^1$ upper function on $[0,T]$. For every $t\in(0,T]$, $x,y\in\mathbb R^d$, and $\mu,\nu\in\mathcal P_2(\mathbb R^d)$, $b$ and $\sigma$ satisfy
		\begin{equation}\label{3.1}
			\begin{aligned}
				&\|b(t,x,\mu)-b(t,y,\nu)\|^2_{\mathbb{R}^d}+\|\sigma(t,x,\mu)-\sigma(t,y,\nu)\|^2_{\mathbb{R}^{d\times m}}\\
				&\leq \alpha_1\big\{\omega(t,\|x-y\|^2_{\mathbb{R}^d})+\omega(t,\mathbb{W}_2^2(\mu,\nu))\big\}+\alpha_2\frac{u'(t)}{u(t)}\big\{\|x-y\|^2_{\mathbb{R}^d}+\mathbb{W}_2^2(\mu,\nu)\big\}.
			\end{aligned}
		\end{equation}
	\end{enumerate}
	
	\begin{thm}\label{thm}
		Assume $\rm(H_1)$--$\rm(H_4)$. Let $X_0\in L^{2l}(\Omega,\mathcal{F},\mathbb{P};\mathbb{R}^d)$ and suppose that $\mathcal V(0,\mathcal L_{X_0})<\infty$. Then equation~\eqref{2.1} has a strong solution $X$ on $[0,T]$ satisfying
		\begin{equation}\label{moment}
			\sup_{0\leq t\leq T}\mathcal V(t,\mathcal L_{X(t)})<\infty.
		\end{equation}
	    Moreover, this solution is unique among all strong solutions satisfying \eqref{moment}.
	\end{thm}
	\begin{proof}
		(i) We first prove pathwise uniqueness. Set $\mu_0:=\mathcal{L}_{X_0}$. Let $X_1$ and $X_2$ be two strong solutions of \eqref{2.1} satisfying \eqref{moment}, with
		$X_1(0)=X_2(0)$ almost surely. Set $\mu_i(t):=\mathcal{L}_{X_i(t)}$ for $i=1,2$, and define
		\[
		M:=\max_{i=1,2}\Big\{\sup_{0\leq t\leq T}\mathcal V(t,\mu_i(t))\Big\}<\infty.
		\]
		
		Assumption $\rm(H_2)$ first yields
		\[
		\mathbb E\int_0^T\big(\|b(r,X_i(r),\mu_i(r))\|_{\mathbb{R}^d}^{2l}
		+\|\sigma(r,X_i(r),\mu_i(r))\|_{\mathbb{R}^{d\times m}}^{2l}\big)dr
		\leq CT(1+M).
		\]
		H\"older's inequality for the drift integral and the Burkholder--Davis--Gundy inequality for the stochastic integral consequently give
		\[
		\mathbb E\big[\sup_{0\leq t\leq T}\|X_i(t)\|_{\mathbb{R}^d}^{2l}\big]
		\leq C_{l,T}\bigl(\mathbb E\|X_0\|_{\mathbb{R}^d}^{2l}+1+M\bigr)<\infty.
		\]
		Set $\Delta X:=X_1-X_2$ and
		\[
		\begin{aligned}
			\Delta b(r)&:=b(r,X_1(r),\mu_1(r))-b(r,X_2(r),\mu_2(r)),\\
			\Delta\sigma(r)&:=\sigma(r,X_1(r),\mu_1(r))-\sigma(r,X_2(r),\mu_2(r)).
		\end{aligned}
		\]
		Then
		\begin{equation}\label{3.3}
			\Delta X(t)=\int_0^t\Delta b(r)\,dr+\int_0^t\Delta\sigma(r)\,dB(r).
		\end{equation}
		H\"{o}lder's inequality and It\^o's isometry yield
		\[
		\mathbb E\|\Delta X(t)\|_{\mathbb{R}^d}^2
		\leq2(T\vee1)\int_0^t\mathbb E\big(\|\Delta b(r)\|_{\mathbb{R}^d}^2+\|\Delta\sigma(r)\|_{\mathbb{R}^{d\times m}}^2\big)\,dr.
		\]
		
		Set $I(t):=\mathbb E\|\Delta X(t)\|_{\mathbb{R}^d}^2$.
		The joint law of $(X_1(r),X_2(r))$ is a coupling, so $\mathbb W_2^2(\mu_1(r),\mu_2(r))\leq I(r)$. Hence Jensen's inequality, the concavity and monotonicity of $\omega$, and \eqref{3.1} give directly
		\begin{equation}\label{3.6}
			I(t)\leq4(T\vee1)\int_0^t\left\{\alpha_1\omega(r,I(r))
			+\alpha_2\frac{u'(r)}{u(r)}I(r)\right\}dr.
		\end{equation}
		For $i=1,2$, Tonelli's theorem, H\"{o}lder's inequality with respect to the probability measure $\mu_i(r)$, $\rm(H_2)$, and the definition of $M$ give
		\begin{equation}\label{3.5}
			\begin{aligned}
				\mathbb{E}\int_{0}^{t}\|b(r,X_i(r),\mu_i(r))\|_{\mathbb{R}^d}^2dr
				&=\int_{0}^{t}\int_{\mathbb{R}^d}\|b(r,x,\mu_i(r))\|_{\mathbb{R}^d}^2\mu_i(r)(dx)dr\\
				&\leq\int_{0}^{t}\left(\int_{\mathbb{R}^d}\|b(r,x,\mu_i(r))\|_{\mathbb{R}^d}^{2l}\mu_i(r)(dx)\right)^{\frac{1}{l}}dr\\
				&\leq\int_{0}^{t}\big[C\big(1+\mathcal{V}(r,\mu_i(r))\big)\big]^{\frac{1}{l}}dr\\
				&\leq t\,[C(1+M)]^{1/l}<\infty.
			\end{aligned}  	
		\end{equation}
		The same estimate holds with $b$ replaced by $\sigma$. Consequently, both coefficient processes belong to $L^2(\Omega\times[0,T])$. For $0\leq s<t\leq T$, \eqref{3.5} and It\^o's isometry show that
		\[
		\begin{aligned}
		\mathbb E\|X_i(t)-X_i(s)\|_{\mathbb{R}^d}^2&\leq 2(t-s)\mathbb E\int_s^t\|b(r,X_i(r),\mathcal L_{X_i(r)})\|_{\mathbb{R}^d}^2dr\\
		&+2\mathbb E\int_s^t\|\sigma(r,X_i(r),\mathcal L_{X_i(r)})\|_{\mathbb{R}^{d\times m}}^2dr\rightarrow0\,\,\hbox{as}\,\,t-s\rightarrow0.
		\end{aligned}
		\]
	    Thus each $X_i$ is continuous as an $L^2$-valued map, and $I$ is continuous. Set $\theta(t):=\frac{I(t)}{u(t)}$ for $t>0$. We claim that $\theta(t)\to0$ as $t\downarrow0$. Indeed, \eqref{3.5} and its analogue for $\sigma$ give a constant $M_T<\infty$, depending only on $C,l,T$, and $M$, such that
		\[
		I(t)\leq2(T\vee1)\int_0^t\mathbb E\big(\|\Delta b(r)\|_{\mathbb{R}^d}^2+\|\Delta\sigma(r)\|_{\mathbb{R}^{d\times m}}^2\big)dr\leq M_Tt.
		\]
		
		Given $\varepsilon>0$, choose $\delta>0$ so that $M_T\leq\varepsilon u'(t)$ on $(0,\delta]$. Then $I(t)\leq\varepsilon\int_0^tu'(r)dr\leq\varepsilon u(t)$ on this interval, proving the claim. We therefore extend $\theta$ continuously to $0$ by $\theta(0)=0$. From \eqref{3.6},
		\[
		\theta(t)\leq4\alpha_1(T\vee1)\int_{0}^{t}\frac{\omega(r,u(r)\theta(r))}{u(r)}dr
		+\frac{4\alpha_2(T\vee1)}{u(t)}\int_{0}^{t}u'(r)\theta(r)dr.
		\]
		Set $\Theta(t):=\sup\limits_{0\leq r\leq t}\theta(r)$. We have
		\[
		\begin{aligned}
			\theta(t)\leq&~4\alpha_1(T\vee1)\int_{0}^{t}\frac{\omega\big(r,u(r)\Theta(r)\big)}{u(r)}dr+\frac{4\alpha_2(T\vee1)}{u(t)}\int_{0}^{t}u'(r)\Theta(r)dr\\
			\leq&~4\alpha_1(T\vee1)\int_{0}^{t}\frac{\omega\big(r,u(T)\Theta(r)\big)}{u(r)}dr+4\alpha_2(T\vee1)\Theta(t).
		\end{aligned}
		\]
		Clearly, the right-hand side of the above inequality is nondecreasing with respect to $t$, and hence
		\[
		\theta(s)\leq4\alpha_1(T\vee1)\int_{0}^{t}\frac{\omega\big(r,u(T)\Theta(r)\big)}{u(r)}dr+4\alpha_2(T\vee1)\Theta(t),\,\,0< s\leq t\leq T.
		\]
		Taking the supremum of the left-hand side over the interval $[0,t]$, we have
		\[
		\Theta(t)\leq4\alpha_1(T\vee1)\int_{0}^{t}\frac{\omega\big(r,u(T)\Theta(r)\big)}{u(r)}dr+4\alpha_2(T\vee1)\Theta(t),\,\,t\in(0,T].
		\]
		
	    Define $\zeta(t):=u(T)\Theta(t)$. Then we have 
		\begin{equation}\label{3.9}
		\zeta(t)\leq \frac{4\alpha_1(T\vee1)u(T)}{1-4\alpha_2(T\vee1)}\int_{0}^{t}\frac{\omega\big(r,\zeta(r)\big)}{u(r)}dr=\alpha\int_{0}^{t}\frac{\omega\big(r,\zeta(r)\big)}{u(r)}dr,\,\,t\in(0,T].
		\end{equation}      
		Clearly $\zeta(0)=0$. We now claim that \eqref{3.9} forces $\zeta=0$ on $(0,T]$. Define $\Psi(t):=\alpha\int_{0}^{t}\frac{\omega(r,\zeta(r))}{u(r)}dr$ with $\Psi(0)=0$. Then $\zeta(t)\leq\Psi(t)$ and
		\[
		\Psi'(t)=F(t,\zeta(t))\leq F(t,\Psi(t)),\,\,t\in[0,T].
		\]
		That is, $\Psi$ is a lower function of the equation 
		\begin{equation}\label{equation.1}
			x'(t)=F(t,x(t))+\varepsilon,\,\,\varepsilon>0
		\end{equation}
		for $t\in[0,T]$ with $x(0)=0$. By ${\rm(H_4)}$, let $\Psi^*$ be an upper function for $F(t,x)$, and we have
		\[
		D_{\pm}\Psi^*(t)>F(t,\Psi^*(t))\,\,\hbox{on}\,\,(0,T).
		\]
		Then, by the continuity of $F$ and $\Psi^*\in C^1$, $\Psi^*$ is an upper function for (\ref{equation.1}) for some $\varepsilon>0$ and $\Psi$ is a lower function for (\ref{equation.1}) on $[0,T]$. Proposition~2.2 of \cite{LZ} therefore gives a solution of \eqref{equation.1}. We denote its maximum solution by $z_\varepsilon$, with $z_\varepsilon(0)=0$. Extend $F(t,x)$ continuously by
		\begin{equation}\label{cases.1}
			F(t,x):=
			\begin{cases}
				F(0,x),&t\in[-T,0],\\
				F(t,x),&t\in(0,T],\\
				F(T,x),&t\in(T,2T].
			\end{cases}
		\end{equation}
	    
		Applying \cite[Proposition 2.3]{LZ} to (\ref{cases.1}), we obtain a function $z_0(t)$ which is the maximum solution of (\ref{equation.1}) for $\varepsilon=0$ on $[0,T]$ and $z_\varepsilon(t)$ converges uniformly to $z_0(t)$ on $[0,T]$ as $\varepsilon\rightarrow0$. Since $\rm (H_4)$ asserts that $z_0(t)\equiv0$ is the unique solution to (\ref{equation.1}) for $\varepsilon=0$, we have
		\[
		0\leq\Psi(t)\leq z_\varepsilon(t)\rightarrow z_0(t)\equiv0\,\,\hbox{as}\,\,\varepsilon\rightarrow0,
		\]
		which proves the claim.
		
		Thus, $X_1(t)=X_2(t)~\mathbb{P}\text{-a.s.}$ for $t\in[0,T]$. $X_1(r)=X_2(r)$ for all rational $r\in[0,T]$, except on some set of probability zero. As $X_1$ and $X_2$ have continuous sample paths almost surely, we have
		\[
		\mathbb{P}\big(\max\limits_{t\in[0,T]}\|X_1(t)-X_2(t)\|_{\mathbb{R}^d}>0\big)=0,
		\] 
		which proves pathwise uniqueness.
		
		(ii) We next prove existence. Set
		\[
		X_0^k:=X_0\mathbf1_{\{X_0\in D_k\}},\,\,
		\mu_0^k:=\mathcal L_{X_0^k}.
		\]
		Since $D_k\uparrow\mathbb R^d$ and $X_0\in L^{2l}$, dominated convergence gives
		\begin{equation}\label{initial-truncation-convergence}
			\mathbb E\|X_0^k-X_0\|^{2l}
			=\mathbb E\big[\|X_0\|^{2l}\mathbf1_{\{X_0\notin D_k\}}\big]\rightarrow0.
		\end{equation}
		In particular, $\mu_0^k\to\mu_0$ in $\mathbb W_2$. Moreover,
		$V_0(0)=0$ and $0\in D_k$ imply
		\[
		\mu_0^k(V_0)
		=\mathbb E\big[V_0(X_0)\mathbf1_{\{X_0\in D_k\}}\big]
		\leq\mu_0(V_0)\leq\mathcal V(0,\mu_0).
		\]
		Hence (\ref{assumption.2.1}) yields
		\begin{equation}\label{initial-lyapunov-bound}
			C_0:=\sup_{k\geq1}\mathcal V(0,\mu_0^k)
			\leq K\bigl(1+\mathcal V(0,\mu_0)\bigr)<\infty.
		\end{equation}
	
		For the localized construction, choose
		$\chi_k\in C_c^\infty(\mathbb R^d;[0,1])$ such that
		$\chi_k=1$ on $D_k$ and
		$\operatorname{supp}\chi_k\subset D_{k+1}$, and set
		\[
		b_k(t,x,\mu):=\chi_k(x)b(t,x,\mu),\,\,
		\sigma_k(t,x,\mu):=\chi_k(x)\sigma(t,x,\mu).
		\]
		For $n\geq1$, let $T_n:=T/n$, $t_i^n:=iT_n$, and
		$\varphi_n(t):=t_i^n$ for $t\in[t_i^n,t_{i+1}^n)$, with the endpoint convention $\varphi_n(T):=T$.
		The smoothly truncated auxiliary coefficients $b_k$ and $\sigma_k$ are used
		only to produce the following one-step weak solution.  Let $\nu$ be
		supported by $\overline D_k$, let $y\in\overline D_k$, and fix the
		$i$-th mesh interval $[t_i^n,t_{i+1}^n]$. Assumption $\rm(H_1)$ makes $b_k$ and $\sigma_k$ bounded on the relevant compact set. Hence the Stroock--Varadhan theorem \cite[Theorem~6.1.7]{STR} yields a weak solution $(X_{i,k}^{n,y,\nu},B_{i,k}^{n,y,\nu})$ of
		\[
		dX_{i,k}^{n,y,\nu}(t)=b_k\bigl(t,X_{i,k}^{n,y,\nu}(t),\nu\bigr)dt+
		\sigma_k\bigl(t,X_{i,k}^{n,y,\nu}(t),\nu\bigr)
		dB_{i,k}^{n,y,\nu}(t),\,\, X_{i,k}^{n,y,\nu}(t_i^n)=y.
		\]
		
		Define
		\[
		\tau_{i,k}^{n,y,\nu}:=\inf\{t\in[t_i^n,t_{i+1}^n]:
		X_{i,k}^{n,y,\nu}(t)\notin D_k\},\,\,
		Y_{i,k}^{n,y,\nu}(t):=
		X_{i,k}^{n,y,\nu}\bigl(t\wedge\tau_{i,k}^{n,y,\nu}\bigr).
		\]
		Since $\chi_k=1$ on $D_k$, this stopped process satisfies
		\begin{equation}\label{absorbed-one-step}
			Y_{i,k}^{n,y,\nu}(t)=y+\int_{t_i^n}^t\mathbf1_{\{r<\tau_{i,k}^{n,y,\nu}\}}b\bigl(r,Y_{i,k}^{n,y,\nu}(r),\nu\bigr)dr+\int_{t_i^n}^t\mathbf1_{\{r<\tau_{i,k}^{n,y,\nu}\}}
			\sigma\bigl(r,Y_{i,k}^{n,y,\nu}(r),\nu\bigr)dB_{i,k}^{n,y,\nu}(r)
		\end{equation}
		for $t\in[t_i^n,t_{i+1}^n]$ and is absorbed after $\tau_{i,k}^{n,y,\nu}$. 
	
	    Applying Lemma~\ref{lem:splicing} with $D=D_k$ and initial law
		$\mu_0^k$ gives continuous adapted processes $(X_k^n,B_k^n)$ on stochastic
		bases $(\Omega^{k,n},\mathcal F^{k,n},(\mathcal F_t^{k,n})_{t\in[0,T]},
		\mathbb P^{k,n})$, where $B_k^n$ is a Brownian motion and $\tau_k^n:=\inf\{t\in[0,T]:X_k^n(t)\notin D_k\}$. Moreover, they satisfy
		\begin{equation}\label{EulerSDE}
		X_k^n(t)=X_k^n(0)+\int_0^t \mathbf{1}_{\{r<\tau_k^n\}}b\bigl(r,X_k^n(r),\nu_k^n(r)\bigr)dr+\int_0^t \mathbf{1}_{\{r<\tau_k^n\}}\sigma\bigl(r,X_k^n(r),\nu_k^n(r)\bigr)dB_k^n(r),
		\end{equation}
		where
		\[
		\mu_k^n(t):=\mathcal L_{\mathbb P^{k,n}}(X_k^n(t)),\,\,
		\nu_k^n(t):=\mu_k^n(\varphi_n(t)).
		\]
		Since the paths are continuous and are absorbed at their first exit, we have
		\begin{equation}\label{euler-activity-identity}
		X_k^n(t)=X_k^n(t\wedge\tau_k^n),\,\,0\leq t\leq T.
		\end{equation}
		
		By the regularity in Lemma~\ref{lem.chain rule}, 
		for a probability measure $\nu$, we write, with a slight abuse of notation,
		\[
	    \mathcal G^\nu\mathcal V(t,z,\mu):=b(t,z,\nu)\cdot\partial_\mu\mathcal V(t,\mu)(z)+\frac12\textnormal{tr}\!\left[\sigma\sigma^\top(t,z,\nu)\partial_z\partial_\mu\mathcal V(t,\mu)(z)\right].
		\]
		The joint chain rule,
		applied to (\ref{EulerSDE}) and then integrated, yields
		\begin{equation}\label{euler-lyapunov-flow}
		\mathcal V(t,\mu_k^n(t))=\mathcal V(0,\mu_0^k)+\int_0^t\left[\partial_r\mathcal V(r,\mu_k^n(r))+\int_{D_k}\mathcal G^{\nu_k^n(r)}\mathcal V(r,z,\mu_k^n(r))\mu_k^n(r)(dz)\right]dr.
		\end{equation}
	
		For fixed $k$, $\rm(H_1)$ gives a bound for $b$ and $\sigma$ on
		$[0,T]\times\overline D_k\times\mathcal P(\overline D_k)$. Consequently,
		for a constant $C_k$ independent of $n$,
		\begin{equation}\label{euler-lag-modulus}
			\sup_{0\leq t\leq T}
			\mathbb W_2^2\bigl(\mu_k^n(t),\nu_k^n(t)\bigr)
			\leq C_k T_n\rightarrow0.
		\end{equation}
		The regularity in the joint chain rule and the continuity in $(x,\mu)$ imply,
		for every fixed $k$,
		\begin{equation}\label{euler-lag-error}
				\varepsilon_{k,n}:=\int_0^T\left|
				\int_{D_k}\left[
				\mathcal G^{\nu_k^n(t)}\mathcal V(t,z,\mu_k^n(t))
				-\mathcal G\mathcal V(t,z,\mu_k^n(t))\right]
				\mu_k^n(t)(dz)\right|dt\rightarrow0.
		\end{equation}
		To see this, for each fixed time, \eqref{euler-lag-modulus} and continuity on the compact state--measure set give convergence of the integrand. The integrand is measurable in time and is uniformly bounded by a constant depending only on $k$. Hence the dominated convergence theorem applies. Then writing \eqref{euler-lyapunov-flow} as the $\mathcal G$-term plus the error
		and using $\rm(H_3)$ gives
		\[
		\mathcal V(t,\mu_k^n(t))
		\leq C_0+\varepsilon_{k,n}
		+\int_0^t\left[
		|\rho(r)|\mathcal V(r,\mu_k^n(r))+|\eta(r)|\right]dr.
		\]
		For every $k$, choose $N_k$ such that $\varepsilon_{k,n}\leq1$ for
		$n\geq N_k$. Gronwall's lemma gives
		\begin{equation}\label{euler-lyapunov-uniform}
			\sup_{k\geq1}\sup_{n\geq N_k}\sup_{0\leq t\leq T}
			\mathcal V(t,\mu_k^n(t))
			\leq C_{\mathcal V}(T),
		\end{equation}
		where
		\[
		C_{\mathcal V}(T):=
		\left(C_0+1+\int_0^T|\eta(r)|dr\right)
		\exp\!\left(\int_0^T|\rho(r)|dr\right).
		\]
		
		For fixed $k$, the family
		$\mathcal{M}_k:=\mathcal P(\overline D_k)$ is compact in $\mathbb W_2$. Apply Lemma~\ref{lem:absorbed-euler-limit} to the tail $n\geq N_k$. It
		produces a weak solution $(\bar X_k,\bar B_k,\bar \theta_k)$ on $(\bar{\Omega}_k,\bar{\mathcal{F}}_k,\bar{\mathbb{P}}_k)$ of
		\begin{equation}\label{delayed-limit-sde}
		\bar X_k(t)=\bar X_k(0)+\int_0^t\mathbf1_{\{r<\bar\theta_k\}}
		b(r,\bar X_k(r),\bar\mu_k(r))dr+\int_0^t\mathbf1_{\{r<\bar\theta_k\}}\sigma(r,\bar X_k(r),\bar\mu_k(r))d\bar B_k(r),
		\end{equation}
		where $\bar\mu_k(t):=\mathcal L_{\bar{\mathbb P}_k}(\bar X_k(t))$, and
		\begin{equation}\label{delayed-limit-stopping}
			\bar X_k(t)=\bar X_k(t\wedge\bar\theta_k),\,\,
			\bar X_k(\bar \theta_k)\in\partial D_k\,\,\hbox{on }\{\bar \theta_k\leq T\}.
		\end{equation}
		By continuity on the compact set, the estimate \eqref{euler-lyapunov-uniform} passes to the following limit
		\begin{equation}\label{truncated-lyap-uniform}
			\sup_{k\geq1}\sup_{0\leq t\leq T}
			\mathcal V(t,\bar\mu_k(t))\leq C_{\mathcal V}(T).
		\end{equation}
		Since $\partial D_k\subset\{V_0=R_k\}$, \eqref{delayed-limit-stopping} and
		\eqref{truncated-lyap-uniform} imply
		\begin{equation}\label{exit-probability}
			\bar{\mathbb P}_k(\bar\theta_k\leq T)
			\leq\frac{\bar{\mathbb E}_kV_0(\bar X_k(T))}{R_k}
			\leq\frac{C_{\mathcal V}(T)}{R_k}\rightarrow0.
		\end{equation}
		Furthermore, $\rm(H_2)$ and \eqref{truncated-lyap-uniform} imply
		\begin{equation}\label{uniform-coefficient-moment}
			\bar{\mathbb E}_k\int_0^T\mathbf1_{\{t<\bar\theta_k\}}
			\left(\|b(t,\bar X_k(t),\bar\mu_k(t))\|^{2l}
			+\|\sigma(t,\bar X_k(t),\bar\mu_k(t))\|^{2l}\right)dt
			\leq CT\bigl(1+C_{\mathcal V}(T)\bigr).
		\end{equation}
		The Burkholder--Davis--Gundy inequality, H\"{o}lder's inequality, and
		\eqref{uniform-coefficient-moment} yield
		\begin{equation}\label{uniform-2l-path}
			\sup_{k\geq1}\bar{\mathbb E}_k
			\left[\sup_{0\leq t\leq T}\|\bar X_k(t)\|_{\mathbb{R}^d}^{2l}\right]<\infty,
		\end{equation}
		and
		\begin{equation}\label{uniform-k-modulus}
			\bar{\mathbb E}_k\|\bar X_k(t)-\bar X_k(s)\|^{2l}
			\leq C_l|t-s|^l,\qquad 0\leq s<t\leq T.
		\end{equation}
		
		For $0\leq s\leq T+2$, define
		\[
		\bar A_k(s):=
		\begin{cases}
			\mathbf1_{\{s<1+\bar\theta_k\}},&\bar\theta_k\leq T,\\
			\mathbf1_{\{s<T+2\}},&\bar\theta_k=\infty.
		\end{cases}
		\]
		For $1\leq s\leq T+1$, let
		\[
		\bar Z_k(s):=\bar X_k(s-1),\,\,\bar W_k(s):=\bar B_k(s-1).
		\]
		The same compactness argument as in the proof of
		Lemma~\ref{lem:absorbed-euler-limit} shows that the laws of $\bar A_k$
		are supported by the compact set $\mathscr A$ defined there. Together with \eqref{uniform-2l-path} and
		\eqref{uniform-k-modulus}, this yields tightness of $\bigl(\bar Z_k,\bar W_k,\bar A_k\bigr)$ on
		\[
		C([1,T+1];\mathbb R^d)\times C([1,T+1];\mathbb R^m)
		\times D([0,T+2];\{0,1\}).
		\]
		Choose a subsequence, still indexed by $k$, such that $\frac{C_{\mathcal V}(T)}{R_k}\leq2^{-k}$. On a common Skorokhod representation space $(\tilde{\Omega},\tilde{\mathcal{F}},\tilde{\mathbb{P}})$, after taking a further subsequence, we have
		\[
		(\tilde Z_k,\tilde W_k,\tilde A_k)
		\rightarrow(\tilde Z,\tilde W,\tilde A)
		\,\,\tilde{\mathbb{P}}\hbox{-a.s.}
		\]
		with uniform convergence in the first two coordinates and $J_1$ convergence
		in the third. Let $\tilde\theta_k$ be the stopping time read from
		$\tilde A_k$. Equality in law and \eqref{exit-probability} give
		\[
		\tilde{\mathbb P}(\tilde\theta_k\leq T)\leq2^{-k}.
		\]
		Hence the Borel--Cantelli theorem yields
		\[
		\lim\limits_{k\rightarrow\infty}\tilde\theta_k=\infty
		\,\,
		\tilde{\mathbb P}\text{-a.s.}
		\]
		Equivalently,
		$\tilde A_k=a_{T+2}$ for all sufficiently large $k$ almost surely.
		Since $\tilde A_k\to\tilde A$ in the $J_1$ topology, we have
		\[
		\tilde A=a_{T+2}
		\,\,
		\tilde{\mathbb P}\text{-a.s.}
		\]
		In particular, for every $t\in[0,T]$,
		\[
		\mathbf1_{\{t<\tilde\theta_k\}}
		\rightarrow1
		\,\,
		\tilde{\mathbb P}\text{-a.s.}
		\]
		We define
		\[
		\tilde X_k(t):=\tilde Z_k(t+1),\,\,
		\tilde B_k(t):=\tilde W_k(t+1)-\tilde W_k(1),
		\]
		and define $(\tilde X,\tilde B)$ in the same way from
		$(\tilde Z,\tilde W)$. Put $\tilde\mu_k(t):=\mathcal L_{\tilde{\mathbb P}}
		(\tilde X_k(t))$ and
		$\tilde\mu(t):=\mathcal L_{\tilde{\mathbb P}}(\tilde X(t))$.
		By \eqref{uniform-2l-path}, the almost sure uniform convergence and Lemma~\ref{lem.2.3}, we have
		\[
		\lim\limits_{k\rightarrow\infty}\tilde{\mathbb{E}}\sup\limits_{t\in[0,T]}\|\tilde{X}_k(t)-\tilde{X}(t)\|^2_{\mathbb{R}^d}=0.
		\]
		Thus we have
		\[
		\sup_{0\leq t\leq T}\mathbb W_2(\tilde\mu_k(t),\tilde\mu(t))
		\rightarrow0.
		\]
		
		The preceding convergence is obtained through a law-based approximation, rather than by imposing pathwise agreement across the domains. Figure~\ref{fig:law-based-truncation-removal} provides a schematic illustration.
		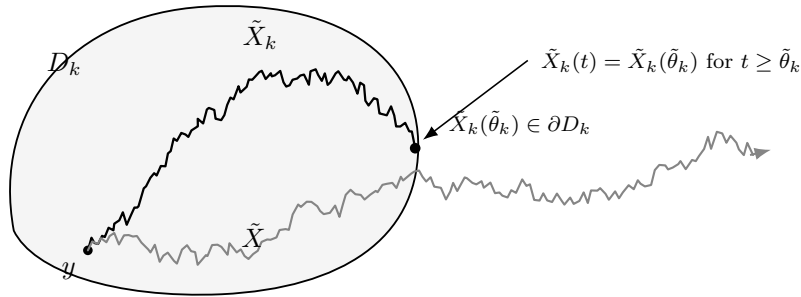
\begin{figure}[H]
			\centering
			\begin{tikzpicture}[
				x=1cm,
				y=1cm,
				font=\small,
				boundary/.style={draw=black,line width=0.7pt},
				stoppedpath/.style={
					draw=black,
					line width=0.85pt,
					line cap=round,
					line join=round
				},
				limitpath/.style={
					draw=black!48,
					line width=0.8pt,
					line cap=round,
					line join=round
				},
				annotation/.style={font=\scriptsize,align=left}
				]
				
				\path[boundary,fill=black!4]
				(-3.80,-0.72)
				.. controls (-4.10,0.90) and (-3.02,2.14) .. (-1.28,2.24)
				.. controls (0.42,2.34) and (1.58,1.63) .. (1.55,0.35)
				.. controls (1.52,-0.82) and (0.56,-1.48) .. (-1.05,-1.56)
				.. controls (-2.48,-1.62) and (-3.54,-1.25) .. cycle;
				
				\node at (-3.13,1.48) {$D_k$};
				
				\coordinate (y0) at (-2.82,-0.98);
				\coordinate (stop) at (1.51,0.37);
				
				\fill (y0) circle (1.8pt);
				\node[below left=1pt] at (y0) {$y$};
				
				\draw[stoppedpath] plot coordinates {
					(-2.820,-0.980) (-2.772,-0.813) (-2.733,-0.883) (-2.685,-0.896)
					(-2.632,-0.788) (-2.596,-0.804) (-2.546,-0.681) (-2.482,-0.715)
					(-2.471,-0.628) (-2.404,-0.583) (-2.357,-0.477) (-2.306,-0.642)
					(-2.247,-0.508) (-2.204,-0.414) (-2.168,-0.415) (-2.150,-0.397)
					(-2.081,-0.161) (-2.050,-0.133) (-1.993,-0.021) (-1.943,0.040)
					(-1.910,0.009) (-1.855,0.135) (-1.794,0.131) (-1.761,0.301)
					(-1.724,0.332) (-1.685,0.499) (-1.642,0.467) (-1.584,0.572)
					(-1.528,0.551) (-1.491,0.519) (-1.450,0.605) (-1.365,0.660)
					(-1.367,0.759) (-1.295,0.697) (-1.246,0.881) (-1.220,0.903)
					(-1.147,0.875) (-1.131,0.760) (-1.056,0.741) (-1.014,0.808)
					(-0.977,0.942) (-0.934,1.039) (-0.892,1.016) (-0.827,1.189)
					(-0.801,1.106) (-0.719,1.182) (-0.690,1.279) (-0.646,1.274)
					(-0.606,1.317) (-0.552,1.377) (-0.519,1.393) (-0.468,1.284)
					(-0.399,1.361) (-0.387,1.320) (-0.344,1.334) (-0.275,1.392)
					(-0.232,1.241) (-0.196,1.185) (-0.151,1.202) (-0.112,1.275)
					(-0.051,1.291) (-0.009,1.331) (0.039,1.343) (0.104,1.209)
					(0.125,1.387) (0.194,1.413) (0.226,1.339) (0.275,1.260)
					(0.315,1.359) (0.352,1.207) (0.387,1.239) (0.458,1.359)
					(0.493,1.387) (0.568,1.276) (0.601,1.267) (0.629,1.319)
					(0.687,1.238) (0.746,1.091) (0.783,1.236) (0.833,1.203)
					(0.851,0.998) (0.893,0.934) (0.964,1.031) (1.010,1.016)
					(1.037,0.830) (1.093,0.816) (1.134,0.805) (1.191,0.923)
					(1.233,0.803) (1.258,0.658) (1.313,0.657) (1.378,0.705)
					(1.431,0.593) (1.466,0.591) (1.510,0.370)
				};
				
				\fill (stop) circle (2pt);
				
				\node[above=2pt] at (-0.54,1.48) {$\tilde X_k$};
				
				\node[annotation,anchor=west] at (1.82,0.70)
				{$\tilde X_k(\tilde\theta_k)\in\partial D_k$};
				
				\draw[-{Latex[length=2mm]},line width=0.6pt]
				(3.00,1.53) -- (1.61,0.48);
				
				\node[annotation,anchor=west] at (3.04,1.52)
				{$\tilde X_k(t)=\tilde X_k(\tilde\theta_k)$ for
					$t\geq\tilde\theta_k$};
				
				\draw[limitpath] plot coordinates {
					(-2.820,-0.980) (-2.761,-0.862) (-2.688,-0.895) (-2.648,-0.818)
					(-2.575,-0.873) (-2.532,-0.791) (-2.471,-0.750) (-2.409,-0.902)
					(-2.342,-0.781) (-2.280,-0.778) (-2.224,-0.797) (-2.171,-0.790)
					(-2.112,-0.877) (-2.021,-0.968) (-1.992,-1.008) (-1.934,-0.906)
					(-1.870,-1.131) (-1.821,-1.050) (-1.763,-1.010) (-1.681,-1.119)
					(-1.642,-1.155) (-1.575,-0.970) (-1.508,-0.912) (-1.454,-0.991)
					(-1.401,-1.003) (-1.368,-1.171) (-1.258,-1.043) (-1.238,-1.055)
					(-1.177,-0.910) (-1.105,-0.905) (-1.053,-0.934) (-0.972,-1.038)
					(-0.947,-0.886) (-0.851,-1.007) (-0.797,-0.967) (-0.747,-1.055)
					(-0.695,-1.094) (-0.634,-1.029) (-0.572,-0.985) (-0.522,-0.868)
					(-0.448,-0.836) (-0.404,-0.840) (-0.325,-0.809) (-0.284,-0.711)
					(-0.204,-0.699) (-0.161,-0.709) (-0.111,-0.506) (-0.038,-0.374)
					(-0.002,-0.406) (0.075,-0.337) (0.137,-0.339) (0.215,-0.397)
					(0.252,-0.459) (0.300,-0.501) (0.382,-0.484) (0.451,-0.388)
					(0.508,-0.537) (0.553,-0.374) (0.591,-0.246) (0.655,-0.179)
					(0.719,-0.238) (0.790,-0.247) (0.840,-0.308) (0.917,-0.243)
					(0.988,-0.194) (1.013,-0.211) (1.077,-0.114) (1.132,-0.101)
					(1.208,-0.088) (1.262,-0.156) (1.320,-0.007) (1.350,-0.056)
					(1.459,0.036) (1.499,0.062) (1.561,0.075) (1.606,-0.006)
					(1.655,-0.070) (1.714,-0.013) (1.762,-0.138) (1.850,-0.153)
					(1.917,-0.307) (1.974,-0.203) (2.012,-0.088) (2.066,-0.196)
					(2.165,-0.016) (2.194,-0.026) (2.267,-0.088) (2.319,-0.108)
					(2.388,-0.066) (2.418,-0.152) (2.505,-0.084) (2.566,-0.224)
					(2.625,-0.215) (2.663,-0.257) (2.732,-0.281) (2.781,-0.185)
					(2.844,-0.220) (2.926,-0.042) (2.982,-0.128) (3.028,-0.107)
					(3.084,-0.313) (3.155,-0.322) (3.205,-0.226) (3.256,-0.320)
					(3.321,-0.341) (3.385,-0.229) (3.449,-0.326) (3.484,-0.375)
					(3.560,-0.282) (3.607,-0.366) (3.682,-0.327) (3.739,-0.214)
					(3.794,-0.354) (3.836,-0.188) (3.920,-0.191) (3.957,-0.158)
					(4.028,-0.233) (4.072,-0.311) (4.155,-0.200) (4.206,-0.292)
					(4.257,-0.214) (4.334,-0.119) (4.402,-0.172) (4.455,-0.053)
					(4.491,-0.099) (4.549,-0.135) (4.631,-0.116) (4.680,0.081)
					(4.739,0.149) (4.819,0.143) (4.855,0.090) (4.917,0.114)
					(4.991,0.160) (5.019,0.200) (5.089,0.093) (5.173,0.215)
					(5.222,0.199) (5.263,0.265) (5.332,0.360) (5.384,0.520)
					(5.424,0.362) (5.501,0.587) (5.572,0.566) (5.627,0.442)
					(5.681,0.535) (5.759,0.443) (5.792,0.355) (5.858,0.279)
					(5.903,0.440) (5.980,0.280)
				};
				
				\draw[-{Latex[length=2.2mm]},draw=black!48,line width=0.8pt]
				(5.94,0.28) -- (6.22,0.35);
				
				\node at (-0.63,-0.78) {$\tilde{X}$};
				
			\end{tikzpicture}
			
			\caption{A schematic of the law-based approximation. The process
			$\tilde X_k$ is absorbed at $\tilde\theta_k$, whereas no pathwise
			agreement between $\tilde X_k$ and $\tilde X$ is imposed. The limiting
			procedure relies on convergence in law, $\mathcal L_{\tilde X_k}\Rightarrow\mathcal L_{\tilde{X}}$ as $k\to\infty$.}
		\label{fig:law-based-truncation-removal}
		\end{figure}
	    $\rm(H_2)$ and
		\eqref{truncated-lyap-uniform} give uniform integrability of the squared
		coefficients. We now identify the limiting martingale problem. Let $(\hat{\mathcal F}_s)_{1\leq s\leq T+1}$ be the usual augmentation of the canonical filtration generated by $(\tilde Z,\tilde W,\tilde A)$, and set
		\[
		\tilde{\mathcal F}_t:=\hat{\mathcal F}_{t+1},
		\,\,0\leq t\leq T.
		\]
		Repeating the preceding canonical-martingale argument, including the zero-quadratic-variation argument for the stochastic-integral identity, and using uniform integrability
		together with $\lim\limits_{k\rightarrow\infty}\tilde\theta_k=\infty$ almost surely, we obtain that $\tilde B$ is a Brownian motion with respect to $(\tilde{\mathcal F}_t)_{0\leq t\leq T}$ and
		\[
		\tilde X(t)=\tilde X(0)
		+\int_0^t b(r,\tilde X(r),\tilde\mu(r))\,dr
		+\int_0^t\sigma(r,\tilde X(r),\tilde\mu(r))\,d\tilde B(r).
		\]
		Moreover, \eqref{initial-truncation-convergence} gives
		$\mathcal L_{\tilde X(0)}=\mu_0$, and Fatou's lemma together with
		\eqref{truncated-lyap-uniform} yields
		\[
		\sup_{0\leq t\leq T}\mathcal V(t,\tilde\mu(t))
		\leq C_{\mathcal V}(T).
		\]
		Thus (\ref{2.1}) has a weak solution $(\tilde{X},\tilde{B})$. We note that the weak solution $(\tilde{X},\tilde{B})$ constructed above
		is compatible. Indeed, on the canonical space the filtration
		$(\tilde{\mathcal F}_t)$ is the augmentation of the filtration generated by $(\tilde X,\tilde B)$, and $\tilde B$ is a Brownian motion with respect to
		$(\tilde{\mathcal F}_t)_{0\leq t\leq T}$. Hence its future increments are independent of $\tilde{\mathcal{F}}_t$, which implies the required compatibility condition (see the argument in \cite[Section~5]{CAR}). Thus, \eqref{2.1} admits a compatible
		$\mathfrak K_{\mathcal V}$-admissible weak solution. Lemma~\ref{lem:yw-lyapunov-class}, together with the pathwise
		uniqueness proved in part~(i), now yields the asserted strong solution
		on the prescribed stochastic basis.
	\end{proof}
	
	\begin{rem}\label{rem:monotonicity}
        (1) The monotonicity condition requires some $\gamma\in L^1(0,T)$ such that
    \begin{equation}\label{standard-monotonicity}
    	2\langle x-y,b(t,x,\mu)-b(t,y,\nu)\rangle
    	+\|\sigma(t,x,\mu)-\sigma(t,y,\nu)\|_{\mathbb{R}^{d\times m}}^2\leq \gamma(t)\bigl(\|x-y\|_{\mathbb{R}^d}^2+\mathbb W_2^2(\mu,\nu)\bigr).
    \end{equation}
    Condition~${\rm(H_4)}$ departs from \eqref{standard-monotonicity} in two essential ways. First, it controls $\|\Delta b\|_{\mathbb{R}^d}^2+\|\Delta\sigma\|_{\mathbb{R}^{d\times m}}^2$ rather than
    $2\langle x-y,\Delta b\rangle+\|\Delta\sigma\|_{\mathbb{R}^{d\times m}}^2$, thereby restricting the coefficient increments directly without exploiting the sign of the drift.  Second, it admits the non-integrable weight $\frac{u'(t)}{u(t)}$ in place of $\gamma\in L^1$.  The example in Section~4 satisfies ${\rm(H_4)}$ but violates \eqref{standard-monotonicity}, and thus separates the Perron--Nagumo-type condition from \eqref{standard-monotonicity}.
    
    The structure of ${\rm(H_4)}$ can be modified in a natural way.  If one keeps
    the Perron--Nagumo right-hand side unchanged but replaces $\|\Delta b\|_{\mathbb{R}^d}^2$
    by the signed inner product $2\langle x-y,\Delta b\rangle$, the condition becomes
    \begin{equation}\label{3.35}
    	\begin{aligned}
    		&2\langle x-y,b(t,x,\mu)-b(t,y,\nu)\rangle
    		+\|\sigma(t,x,\mu)-\sigma(t,y,\nu)\|_{\mathbb{R}^{d\times m}}^2\\
    		&\leq \alpha_1\bigl\{\omega(t,\|x-y\|_{\mathbb{R}^d}^2)+\omega(t,\mathbb W_2^2(\mu,\nu))\bigr\}
    		+\alpha_2\frac{u'(t)}{u(t)}\bigl\{\|x-y\|_{\mathbb{R}^d}^2+\mathbb W_2^2(\mu,\nu)\bigr\}.
    	\end{aligned}
    \end{equation}
    It\^o's formula yields a similar estimate to~\eqref{3.6}, but without the factor $T\vee1$. Accordingly, it suffices to assume
    $\alpha_2<\frac{1}{2}$ and to take $\alpha=\frac{2\alpha_1u(T)}{1-2\alpha_2}$. The remainder of the uniqueness proof carries over without any further change.
    
    Condition~\eqref{3.35} occupies an intermediate position between
    \eqref{standard-monotonicity} and ${\rm(H_4)}$.  It shares the inner-product
    structure with \eqref{standard-monotonicity} and therefore handles dissipative
    drifts particularly well, yet it inherits the non-integrable singular weight
    from ${\rm(H_4)}$. Conversely, ${\rm(H_4)}$ imposes a direct bound on $\|\Delta b\|_{\mathbb{R}^d}^2$ and is therefore independent of the sign of $\langle x-y,\Delta b\rangle$.
     
    (2) In the path-dependent framework of~\cite{RP}, the local Lipschitz condition allows one to truncate the coefficients on bounded subsets of the path space, so that the truncated equation recovers the Lipschitz and monotonicity conditions required for the Borel--Cantelli argument.  In the present paper, the Lipschitz and monotonicity conditions are replaced by the Perron--Nagumo condition ${\rm(H_4)}$. The path-space truncation strategy of~\cite{Liu-Ma,RP} therefore does not apply here. This explains why we rely on weak convergence, which also differs from the classical methods.
    
    \end{rem}
	
	\section{Example}
	
	We give an explicit MVSDE that satisfies Theorem~\ref{thm} but falls outside the Lipschitz, Osgood, and monotonicity conditions.
	
	For $T=\frac{1}{2}$ and $d=m=1$, we consider the following coefficients:
	\begin{equation}\label{4.1}
		b_1(t,x):=
		\begin{cases}
			0, &t\in[0,\frac{1}{2}],\,x\in(-\infty,0),\\
			\frac{\sqrt{6}}{2}xt^\frac{-1}{4},&t\in(0,\frac{1}{2}],\,x\in[0,\frac{2}{3}t),\\
			x^\frac{1}{2}t^\frac{1}{4}, &t\in[0,\frac{1}{2}],\,x\in[\frac{2}{3}t,\frac{\sqrt{2}}{2}],\\
			(\frac{t}{2})^\frac{1}{4},&t\in[0,\frac{1}{2}],\,x\in(\frac{\sqrt{2}}{2},\infty),	
		\end{cases}
	\end{equation}
	and 
	\begin{equation}\label{4.2}
		\sigma_1(t,x):=
		\begin{cases}
			0, &t\in[0,\frac{1}{2}],\, x\in(-\infty,0),\\
			\frac{x}{5\sqrt{t}}, &t\in(0,\frac{1}{2}],\,x\in[0,t),\\
			\frac{\sqrt{t}}{5}, &t\in[0,\frac{1}{2}],\,x\in[t,\infty).
		\end{cases}
	\end{equation}
	As in the example in \cite{LZ}, the function $x\mapsto b_1(t,x)$ is concave and nondecreasing on $[0,\infty)$ for every $t\in[0,\frac{1}{2}]$. We have 
	\[
	\big|b_1(t,x)-b_1(t,y)\big|\leq b_1(t,|x-y|),~x,y\in\mathbb{R}.
	\]
	Define the function $\psi:[0,\frac{1}{2}]\times[0,\infty)\rightarrow[0,\infty)$ by
	\[
	\psi(t,x):=
	\begin{cases}
		\sqrt{x}, &t=0,\,x\in[0,\frac{1}{2}],\\[2pt]
		\frac{\sqrt{2}}{2}, &t=0,\,x\in(\frac{1}{2},\infty),\\[2pt]
		\frac{3x}{2t}, &t\in(0,\frac{1}{2}],\,x\in[0,\frac{4t^2}{9}),\\[2pt]
		\sqrt{x}, &t\in(0,\frac{1}{2}],\,x\in[\frac{4t^2}{9},\frac{1}{2}],\\[2pt]
		\frac{\sqrt{2}}{2}, &t\in(0,\frac{1}{2}],\,x\in(\frac{1}{2},\infty).
	\end{cases}
	\]
	Then we have
	\[
	\big|b_1(t,x)-b_1(t,y)\big|^2\leq b_1^2(t,|x-y|)=\sqrt{t}\psi(t,|x-y|^2)=\omega(t,|x-y|^2),~x,y\in\mathbb{R},
	\] 
	where the function $\omega:[0,\frac{1}{2}]\times[0,\infty)\rightarrow[0,\infty)$ is defined by $\omega(t,x):=\sqrt{t}\psi(t,x)$. For every $t$, the function $x\mapsto\psi(t,x)$ is concave and nondecreasing. By \cite[Proposition 1]{AUG}, $\lambda\psi(t,x)$ satisfies the Perron condition for $\lambda\in(0,\frac{4}{3})$ and satisfies neither the Lipschitz condition nor the Osgood condition. 
	
	Let $h:\mathbb{R}\rightarrow\mathbb{R}$ be a 1-Lipschitz function, and define $F_C:\mathbb{R}\rightarrow\mathbb{R}$ by
	\[
	F_C(x):=
	\begin{cases}
		|x|, &|x|\leq C,\\
		C, &|x|>C,
	\end{cases}
	\]
	where $C>0$ is a constant. Define a continuous function $f:\mathcal{P}_2(\mathbb{R})\rightarrow\mathbb{R}$ by
	$f(\mu):=h\big(\int_{\mathbb{R}}F_C(x)\mu(dx)\big)$. $f$ is continuous with respect to weak convergence, and hence also with respect to the $\mathbb{W}_2$-topology. Thus we have
	\[
	|f(\mu)-f(\nu)|=\big|h\big(\int_{\mathbb{R}}F_C(x)\mu(dx)\big)-h\big(\int_{\mathbb{R}}F_C(y)\nu(dy)\big)\big|\leq\big|\int_{\mathbb{R}}F_C(x)\mu(dx)-\int_{\mathbb{R}}F_C(y)\nu(dy)\big|
	\]
	for every $\mu,\nu\in\mathcal{P}_2(\mathbb{R})$. Then, for any $\gamma\in\Gamma(\mu,\nu)$, where $\Gamma(\mu,\nu)$ stands for the set of all couplings for $\mu$ and $\nu$, we have 
	\[
	\big|f(\mu)-f(\nu)\big|\leq\big|\int_{\mathbb{R}\times\mathbb{R}}\big(F_C(x)-F_C(y)\big)\gamma(dx,dy)\big|\leq\Big(\int_{\mathbb{R}\times\mathbb{R}}\big|F_C(x)-F_C(y)\big|^2\gamma(dx,dy)\Big)^{\frac{1}{2}}.
	\]
	By the definition of $F_C(x)$, we have
	\[
	|F_C(x)-F_C(y)|\leq|x-y|,
	\,\,x,y\in\mathbb R.
	\]
	This implies
	\begin{equation}\label{4.3}
		\big|f(\mu)-f(\nu)\big|\leq\Big(\inf\limits_{\gamma\in\Gamma(\mu,\nu)}\int_{\mathbb{R}\times\mathbb{R}}|x-y|^2\gamma(dx,dy)\Big)^{\frac{1}{2}}=\mathbb{W}_2(\mu,\nu).
	\end{equation}
	Define $b_2(t,\mu):=b_1(t,f(\mu))$, fix constants $a_1,a_2>0$, and set
	\begin{equation}\label{4.4}
		b(t,x,\mu):=\frac{\sqrt{2a_1}}{2}\big(b_1(t,x)+b_2(t,\mu)\big).
	\end{equation}
	Thus, we have
	\[
	\begin{aligned}
		\big|b(t,x,\mu)-b(t,y,\nu)\big|^2&=\frac{a_1}{2}\big|b_1(t,x)-b_1(t,y)+b_2(t,\mu)-b_2(t,\nu)\big|^2\\
		&\leq a_1\big|b_1(t,x)-b_1(t,y)\big|^2+a_1\big|b_1(t,f(\mu))-b_1(t,f(\nu))\big|^2\\
		&\leq a_1\big\{\omega(t,|x-y|^2)+\omega(t,|f(\mu)-f(\nu)|^2)\big\}\\
		&\leq a_1\big\{\omega(t,|x-y|^2)+\omega\big(t,\mathbb{W}_2^2(\mu,\nu)\big)\big\},
	\end{aligned}
	\]
	for $x,y\in\mathbb{R}$ and $\mu,\nu\in\mathcal{P}_2(\mathbb{R})$. 
	
	Furthermore, it follows from \cite{NEG} that
	\[
	|\sigma_1(t,x)-\sigma_1(t,y)|^2\leq\frac{1}{25t}|x-y|^2\leq\frac{1}{8}\frac{u'(t)}{u(t)}|x-y|^2,~t\in(0,\frac{1}{2}],~x,y\in\mathbb{R},
	\]
	where $u(t)=\sqrt{t}$. If $\sigma_1$ were Lipschitz with constant $L>0$, then, for every $t\in(0,\frac{1}{2}]$,  
	\[
	|\sigma_1(t,t)-\sigma_1(t,0)|=\frac{\sqrt{t}}{5}\leq Lt.
	\]
	Hence $t\geq\frac{1}{25L^2}$, which is impossible for sufficiently small $t$. We then define the continuous function $\sigma_2:[0,\frac{1}{2}]\times\mathcal{P}_2(\mathbb{R})\rightarrow\mathbb{R}$ by $\sigma_2(t,\mu):=\sigma_1(t,f(\mu))$ and the coefficient $\sigma:[0,\frac{1}{2}]\times\mathbb{R}\times\mathcal{P}_2(\mathbb{R})\rightarrow\mathbb{R}$ by
	\begin{equation}\label{4.5}	
		\sigma(t,x,\mu):=\frac{\sqrt{2a_2}}{2}\Big(\sigma_1(t,x)+\sigma_2(t,\mu)\Big).
	\end{equation}
	Thus, we have
	\[
	\begin{aligned}
		\big|\sigma\big(t,x,\mu\big)-\sigma\big(t,y,\nu\big)\big|^2&=\frac{a_2}{2}\big|\sigma_1(t,x)-\sigma_1(t,y)+\sigma_2(t,\mu)-\sigma_2(t,\nu)\big|^2\\
		&\leq a_2\big|\sigma_1(t,x)-\sigma_1(t,y)\big|^2+a_2\big|\sigma_1(t,f(\mu))-\sigma_1(t,f(\nu))\big|^2\\
		&\leq\frac{a_2}{8}\frac{u'(t)}{u(t)}\Big\{|x-y|^2+\mathbb{W}_2^2(\mu,\nu)\Big\}.
	\end{aligned}
	\]
	Consequently, ${\rm(H_4)}$ holds with $\alpha_1=a_1$ and $\alpha_2=a_2/8$, provided $0<a_2<2$ and
	\[
	\alpha=\frac{4\sqrt{2}a_1}{2-a_2}\in\Big[0,\frac43\Big).
	\]
	An upper function required in ${\rm(H_4)}$ may be chosen as $\psi^*(t)=ct$ with $c>\frac{\sqrt{2}\alpha}{2}$ because $\psi(t,x)\leq\sqrt2/2$.
	
	This example also fails the standard monotonicity condition \eqref{standard-monotonicity}. Indeed, fix $\mu=\nu$, set $y=0$ and $x=t$, and let $t\downarrow0$. Since $b_1(t,\cdot)$ is nondecreasing, the drift contribution on the left-hand side of \eqref{standard-monotonicity} is nonnegative, whereas
	\[
	\big|\sigma(t,t,\mu)-\sigma(t,0,\mu)\big|^2
	=\frac{a_2}{2}\big|\sigma_1(t,t)-\sigma_1(t,0)\big|^2
	=\frac{a_2}{50}t.
	\]
	Thus \eqref{standard-monotonicity} would force $\gamma(t)\geq a_2/(50t)$, which is impossible for $\gamma\in L^1(0,\frac12)$. Nevertheless, the normalized Perron--Nagumo argument applies and yields pathwise uniqueness among strong solutions satisfying \eqref{moment}.
	
	Finally, $b$ and $\sigma$ are bounded and continuous, so ${\rm(H_1)}$ holds. With $V_0(x):=|x|^2$, $V(t,x,\mu):=|x|^2+\mu(|\cdot|^2)$, ${\rm(H_2)}$ follows immediately. Moreover, Young's inequality and boundedness of the coefficients give ${\rm(H_3)}$ uniformly in $k$. Hence, for every $X_0\in L^{2l}$ and $\mathcal{V}(0,\mathcal{L}_{X_0})<\infty$, Theorem~\ref{thm} yields a strong solution satisfying \eqref{moment}. This solution is unique among strong solutions satisfying \eqref{moment}, although none of the Lipschitz, Osgood, or monotonicity conditions holds.
	
	\section*{Acknowledgements}
	
	This work is supported by the National Key R\&D Program of China (No. 2023YFA1009200),
	NSFC (Grants 12531009 and 11925102).

\end{document}